\RequirePackage{etoolbox}
\csdef{input@path}{{style/}{graphics/}}
\documentclass[aap,MSNbibl,seceqn,dvips]{arximspdf}
\usepackage{mathbh}
\usepackage{url,breakurl}

\doi{10.1214/14-AAP1034} 
\volume{25}
\issue{3}
\pubyear{2015}
\firstpage{1686}
\lastpage{1727}

\makeatletter
\def\mid{|}
\def\widetilde{\tilde}

\def\dmin{d_{\min}}

\def\EE{\mathbb{E}}
\def\PP{\mathbb{P}}
\def\NN{\mathbb{N}}

\def\RR{\mathbb{R}}

\def\ind{\mathbh{1}}

\def\cF{\mathcal{F}}
\def\cJ{\mathcal{J}}
\def\o{{\varnothing}}

\def\tod{\stackrel{d}{\rightarrow}}
\def\top{\stackrel{p}{\rightarrow}}

\def\tq{\tilde{q}}

\def\dist{\operatorname{dist}}
\def\diam{\operatorname{diam}}
\def\flood{\mathrm{flood}}

\def\tx{\mathrm{tx}}
\def\Bin{\operatorname{Bin}}

\def\td{\widetilde{d}}

\def\R{{\mathbb R}}

\def\N{{\mathbb N}}

\def\cF{{\mathcal F}}
\def\hd{\widehat{d}}
\def\hS{\widehat{S}}
\def\upi{\underline{\pi}}
\def\um{\underline{m}}

\def\uD{\underline{D}}
\def\uphi{\underline{\phi}}
\def\bpi{\overline{\pi}}
\def\bm{\overline{m}}

\def\bD{\overline{D}}
\def\bT{\overline{T}}

\def\bnu{\overline{\nu}}

\def\bT{\overline{T}}
\def\hD{\widehat{D}}
\def\hS{\widehat{S}}

\def\tq{\tilde{q}}
\def\tT{\widetilde{T}}
\def\hp{\hat{p}}

\def\Var{\operatorname{Var}}

\def\Bin{\operatorname{Bin}}

\def\cF{\mathcal{F}}

\def\Core{\operatorname{Core}}

\def\cF{\mathcal{F}}
\def\cE{\mathcal{E}}
\def\cC{\mathcal{C}}
\def\td{{\tilde{d}}}
\def\tG{{\tilde{G}}}
\def\tT{\widetilde{T}}
\def\tB{\widetilde{B}}
\def\td{\tilde{d}}
\def\deg{\operatorname{deg}}
\def\HH{\mathcal{H}}

\newtheorem{theorem}{Theorem}[section]
\newtheorem{lemma}[theorem]{Lemma}
\newtheorem{corollary}[theorem]{Corollary}
\newproclaim{remark}[theorem]{Remark}
\newtheorem{proposition}[theorem]{Proposition}
\newproclaim{condition}[theorem]{Condition}
\makeatother

\begin{document}
\begin{frontmatter}

\title{The diameter of weighted random graphs}
\runtitle{Weighted diameter}

\begin{aug}
\author[A]{\fnms{Hamed}~\snm{Amini}\corref{}\ead[label=e1]{Hamed.Amini@epfl.ch}}
\and
\author[B]{\fnms{Marc}~\snm{Lelarge}\ead[label=e2]{Marc.Lelarge@ens.fr}\thanksref{T1}}
\runauthor{H. Amini and M. Lelarge}
\thankstext{T1}{Supported by the French Agence Nationale de la
Recherche (ANR) under reference ANR-11-JS02-005-01 (GAP project).}
\affiliation{EPFL and INRIA}
\address[A]{\'Ecole Polytechnique F\'ed\'erale de Lausanne\\
Quartier UNIL-Dorigny, Extranef 249\\
1015 Lausanne\\
Switzerland\\
\printead{e1}}
\address[B]{DYOGENE, INRIA\\
\'Ecole Normale Superi\'eure\\
23 avenue d'Italie\\
75214 Paris Cedex 13\\
France\\
\printead{e2}}
\end{aug}

\received{\smonth{8} \syear{2012}}
\revised{\smonth{12} \syear{2013}}

%
\begin{abstract}
In this paper we study the impact of random exponential edge weights on
the distances in a random graph and, in particular, on its diameter.
Our main result consists
of a precise asymptotic expression for the maximal weight of the
shortest weight paths between all vertices (the weighted diameter) of
sparse random graphs, when the edge weights are i.i.d. exponential
random variables.
\end{abstract}

%
\begin{keyword}[class=AMS]
\kwd[Primary ]{60C05}
\kwd{05C80}
\kwd[; secondary ]{90B15}
\end{keyword}

\begin{keyword}
\kwd{First-passage percolation}
\kwd{weighted diameter}
\kwd{random graphs}
\end{keyword}
\end{frontmatter}

\section{Introduction and main results}\label{sec:fpp-intro}

Real-world networks are described not only by their graph structure,
which give us information about valid links between vertices in the
network, but also by their associated edge weights, representing cost
or time required to traverse the edge.
The analysis of the asymptotics of typical distances in edge weighted
graphs has received much interest by the statistical physics community
in the context of \emph{first-passage percolation
problems}. First-passage percolation (F.P.P.) describes the dynamics of
a fluid spreading within a random medium. In this paper we study the
impact of random exponential edge weights on the distances in a random
graph and, in particular, on its diameter.

The typical distance and diameter of nonweighted graphs have been
studied by many people, for various models of random graphs. A few
examples are the results of Bollob\'as and Fernandez de la Vega~\cite
{Boldlv}, van der Hofstad, Hooghiemstra and Van Mieghem~\cite{HHM05},
Fernholz and Ramachandran~\cite{fern07}, Chung and Lu~\cite{CL03},
Bollob\'as, Janson and Riordan~\cite{boljanrio} and Riordan and
Wormald~\cite{riowor10}.
The first-passage percolation model has been mainly studied on lattices
motivated by its subadditive property and its link to a number of other
stochastic processes; see, for example, \cite
{grimkest84,kesten86,hagpem98} for a more detailed discussion.
First-passage percolation with exponential weights has received
substantial attention (see, e.g., \cite{amperes,bhamidi08,HHM05,BHH10,BHH09,janson99,BHH10extreme}), in particular on
the complete graph and more recently, also on random graphs.

A weighted graph $(G,w)$ is the data of a graph $G=(V,E)$ and a collection
of weights $w=\{w_e\}_{e\in E}$ associated to each edge $e\in E$. We
suppose that all the edge weights are nonnegative.
For two vertices $a$ and $b\in V$, a path between $a$ and $b$ is a sequence
$\pi=(e_1,e_2,\ldots,e_k)$ where $e_i=\{v_{i-1},v_i\}\in E$ and $v_i\in V$
for $i\in\{1,\ldots,k\} = [1,k]$, with $v_0=a$ and $v_k=b$. We write
$e\in\pi$ if the edge $e\in E$ belongs to the path $\pi$, that is, if
$e=e_i$ for an $i \in[1,k]$.
For $a,b \in V$, the weighted distance between $a$ and $b$ is given by
\[
\dist_w(a,b) =\dist_w(a,b;G) = \min
_{\pi\in\Pi(a,b)}\sum_{e\in\pi} w_e,
\]
where the minimum is taken over all the paths between $a$ and $b$ in the
graph $G$. The \emph{weighted diameter} is then given by
\[
\diam_w(G) = \max\bigl\{\dist_w(a,b), a,b\in V,
\dist_w(a,b) < \infty\bigr\},
\]
and the \emph{weighted flooding time} for $a\in V$ is defined by
\[
\flood_w(a,G) = \max\bigl\{\dist_w(a,b), b\in V,
\dist _w(a,b) < \infty\bigr\}.
\]
%

\subsection{Random graphs with given degree sequence}
For $n\in\N$, let $(d_i)_1^n$ be a sequence of nonnegative integers
such that $\sum_{i=1}^n d_i$ is even. By means of the configuration
model (Bender and Canfield \cite{bencan78}, Bollob{\'a}s \cite{bol80}),
we define a random multigraph with given degree sequence $(d_i)_1^n$,
denoted by
$G^*(n,(d_i)_1^n)$ as follows: to each node $i\in[1,n]$ we associate
$d_i$ labeled half-edges. All half-edges need to be paired to construct
the graph; this is done by uniformly matching them. When a half-edge of
$i$ is paired with a half-edge of $j$, we interpret this as an edge
between $i$ and $j$. The graph $G^*(n,(d_i)_1^n)$ obtained following
this procedure may not be simple, that is, may contain self-loops due
to the pairing of two half-edges of $i$, and multi-edges due to the
existence of more than one pairing between two given nodes. Conditional
on the multigraph $G^*(n,(d_i)_1^n)$ being a simple graph, we obtain a
uniformly distributed random graph with the given degree sequence,
which we denote by $G(n, (d_i)_1^n)$, \cite{janson09b}.
We consider asymptotics as the numbers of vertices tend to infinity, and
thus we assume throughout the paper that we are given, for each $n$, a
sequence $\mathbf{d}^{(n)} = (d^{(n)}_i)_1^n = (d_i)_1^n$ of nonnegative
integers such that $\sum_{i=1}^n d^{(n)}_i$ is even. For notational
simplicity we will sometimes not show the dependency on $n$ explicitly.

For $k\in\N$, let $u_k^{(n)}=|\{i, d_i=k\}|$ be the number of vertices
of degree $k$.
From now on, we assume that the sequence $(d_i)_1^n$
satisfies the following regularity conditions analogous to the ones
introduced in \cite{MolReed98}: 

\begin{condition}\label{cond-dil}
For each $n$, $\mathbf{d}^{(n)} = (d^{(n)}_i)_1^n = (d_i)_1^n$ is a
sequence of positive
integers such that $\sum_{i=1}^n d_i$ is even, and for some probability
distribution $(p_r)_{r=1}^{\infty}$ over integers independent of $n$
and with finite mean $\mu:=\sum_{k\geq1}kp_k\in[1,\infty)$, the
following holds:
\begin{longlist}[(ii)]
\item[(i)] $u_k^{(n)}/n\to p_k$ for every $k\geq1$ as $n \to\infty$;
%
\item[(ii)] for some $\varepsilon>0$, $\sum_{i=1}^n d_i^{2+\varepsilon}=O(n)$.
\end{longlist}
\end{condition}

Note that the condition $d_i \geq1$ for all $i$ is not restrictive
since removing all isolated vertices from a graph will not affect the
(weighted) distances.


\subsection{Main results}
We define $q=\{q_k\}_{k=0}^{\infty}$ the size-biased probability mass
function corresponding to $p$ by
%
\begin{equation}
\label{eq:defq} \forall k\geq0\qquad q_k:= \frac{(k+1)p_{k+1}}{\mu},
\end{equation}
and let $\nu$ denote its mean
%
\begin{equation}
\label{eq:defnu}\nu:= \sum_{k=0}^{\infty}
kq_k\in(0,\infty)\qquad \bigl[\mbox{by Condition~\ref{cond-dil}(ii)}\bigr].
\end{equation}

Let $\phi_p(z)$ be the probability generating function of $\{p_k\}
_{k=0}^{\infty}\dvtx\phi_p(z) =\break   \sum_{k=0}^{\infty} p_k z^k$, and let
$\phi_q(z)$ be the probability generating function of $\{q_k\}
_{k=0}^{\infty}\dvtx\break \phi_q(z) = \sum_{k=0}^{\infty} q_k z^k=\phi
'_p(z)/\mu$.
In this paper, we will consider only the case where $\nu>1$. In
particular, there exists a unique $\lambda$ in $(0,1)$ such that
$\lambda= \phi_q(\lambda)$, and if $\mathcal{C}$ is the size
(in number of vertices) of the largest
component of $G(n, (d_i)_1^n)$, then we have by Molloy and
Reed~\cite{MolReed98} and Janson and Luczak~\cite{janson09b},
$\mathcal{C}/n \stackrel{p}{\rightarrow} 1- \phi_p(\lambda) > 0$.
In addition, we introduce
%
\begin{equation}
\lambda_* = \phi'_q(\lambda) = \sum
_{k=1}^{\infty} k q_k \lambda^{k-1}
\in[0,1).
\end{equation}
%

We can now state our main theorem.

\begin{theorem}\label{thm-fpp-main}
Let $(G(n, (d_i)_1^n),w)$ be a sequence of random weighted graphs where
$w=\{w_e\}_{e\in E}$ are i.i.d. rate one exponential random variables.

Assume {Condition~\ref{cond-dil}} and that $\nu$ defined in
(\ref{eq:defnu}) is such that $\nu>1$.

Assume that all the graphs have the same
minimum degree denoted by
$\dmin=\min_{i\in[1,n]}d_i$ and moreover that $p_{\dmin}>0$.
Let $\Gamma\dvtx \N^*\to\R$ be defined by
%
\begin{equation}
\label{def:Gamma}\Gamma(d):= d \ind[d \geq3]+ 2(1 - q_1) \ind[d = 2]
+ (1-\lambda_*) \ind[d = 1].
\end{equation}
Let a, b be two uniformly chosen vertices in this graph. If we
condition the vertices $a$ and $b$ to be connected, we have
%
\begin{equation}
\label{th:dist}\frac{\dist_w(a,b;G(n, (d_i)_1^n))}{\log n}
\stackrel{p} {\rightarrow}
\frac{1}{\nu-1}.
\end{equation}
If we condition the vertex $a$ to be in the largest component, we have
%
\begin{equation}
\label{th:flood}\frac{\flood_w(a,G(n, (d_i)_1^n))}{\log n}
\stackrel{p} {\rightarrow}
\frac{1}{\nu-1} +\frac{1}{\Gamma(\dmin)}.
\end{equation}
Finally, we have
%
\begin{equation}
\label{th:diam}\frac{\diam_w(G(n, (d_i)_1^n))}{\log n}
\stackrel {p} {\rightarrow}
\frac{1}{\nu-1} +\frac{2}{\Gamma(\dmin)}.
\end{equation}
\end{theorem}

\begin{remark}
Note that $\nu>1$ implies that $\sum_{k=0}^{\infty}k(k-2)p_k > 0$ so
that there is a
positive fraction of nodes in $G(n, (d_i)_1^n)$ with degree $3$ or
larger. In particular, we have $q_1=2p_2/\mu<1$ and $\lambda_*<1$ so
that we have $\Gamma(d)>0$ for all $d\in\N^*=\{1,2,\ldots\}$.
\end{remark}

We now comment our result with respect to related literature. Our main
contribution is (\ref{th:diam}) while results (\ref{th:dist}) and (\ref
{th:flood}) follow from the analysis required to prove (\ref{th:diam}).
Indeed, a much stronger version of (\ref{th:dist}) has been proved for
a slightly different model of random graphs by Bhamidi, van der Hofstad
and Hooghiemstra in \cite{BHH09}. Theorem~3.1 in \cite{BHH09} shows
that if the sequence $(d_i)_1^n$ is a sequence of i.i.d.
(nondegenerate) random variables with $d_{\min}\geq2$ and finite
variance, then there exists a random variable $V$ such that
(conditioning on $a$ and $b$ being connected)
\[
\dist_w\bigl(a,b;G^*\bigl(n, (d_i)_1^n
\bigr)\bigr) - \frac{\log n}{\nu-1} \stackrel {d} {\rightarrow} V.
\]
We expect this result to be valid for our model of random graphs
$G(n,(d_i)_1^n)$ where the degrees $d_i$ satisfy Condition~\ref
{cond-dil} (but we did not try to prove it). Bhamidi, van der Hofstad
and Hooghiemstra \cite{BHH09,BHH10extreme} also give results when the degree sequence has no finite
second moment and no finite first moment.

Motivated by the analysis of the diameter of the largest component of a
critical Erd\H{o}s--R\'enyi random graph (without edge weights), Ding
et al. \cite{ding09} show that if $d_i=r\geq3$ for all $i$, then we
have with high probability
\[
\diam_w\bigl(G^*(n,r)\bigr)= \biggl( \frac{1}{r-2}+
\frac{2}{r} \biggr)\log n+O(\log \log n).
\]

The intuition behind this formula is simple: consider a vertex in
$G^*(n,r)$; its closest neighbor is at distance given by
an exponential random variable with rate $r$ (i.e., the minimun of
$r$ exponential rate one random variables). Hence the probability
for this distance to be larger than $\log n /r$ is $n^{-1}$. Since
there are $n$ vertices with degree $r$, a simple
argument shows that we will find two nodes with closest neighbors at
distance $\log n /r$. The diameter will
be obtained by taking a shortest path between these two nodes. Each
such node will first give a contribution of $\log n /r$ to reach
its closest neighbor and then the path between these neighbors will be
typical, of the order $\log n /(r-2)$.
This simple heuristic argument shows that our result on the diameter
depends crucially on the weights being exponentially distributed or at
least have an exponential tail. We refer to \cite
{bhamidi2012universality} for recent results on distances with i.i.d.
weights. As we will see, the presence of nodes with degree one and two
makes the analysis much more involved than in \cite{ding09}. As soon as
a fraction of nodes have degree two, there will be long paths
constitued by a chain of such nodes, and we will see that these paths
contribute to the diameter.

In \cite{amdrle-fl}, this result is used to analyze an asynchronous
randomized broadcast algorithm for random regular graphs. In
continuous-time, each node is endowed with a Poisson point process with rate
1 and contacts one of its neighbors uniformly at random at each point
of his process. In a push model, if a node holds the message, it
passes the message to its randomly chosen neighbor regardless of its
state. The results in \cite{amdrle-fl} show that the asynchronous
version of the algorithm performs better than its synchronized version:
in the large size limit of the graph, it will reach the whole network
faster even if the local dynamics are similar on average. 

We end this section by a simple remark.
Our results can be applied to some other random graphs models too by
conditioning on the degree sequence. In particular, our results will
apply whenever the random graph conditioned on the degree sequence has
a uniform distribution over all possibilities. Notable examples of such
graphs are $G(n,p)$, the Bernoulli random graph with $n$ vertices and
edge probability $p$ and $G(n,m)$, the uniformly random graph with $n$
vertices and $m$ edges. For example, for $G(n,p)$ with $np\to\mu\in
(0,\infty)$ or $G(n,m)$ with $2m/n\to\mu$, Condition~\ref{cond-dil}(i)
holds in probability with $(p_k)$, a Poisson distribution with
parameter~$\mu$, $p_k = e^{-\mu}\frac{\mu^k}{k!}$.
In Appendix \ref{sec:Sko}, we show that thanks to Skorohod's coupling
theorem \cite{kallenberg}, Theorem~3.30, our results still apply in this
setting. By taking care of removing isolated nodes, our result gives in
this case [note that $\phi_q(z) = e^{-\mu(1-z)}$].

\begin{theorem}
Let $\mu>1$ be fixed, and let $\lambda_*<1$ satisfy $\lambda
_*e^{-\lambda_*} = \mu e^{-\mu}$. Assume $G_n=G(n,p)$ where $np\to\mu
\in(0,\infty)$ [or $G_n=G(n,m)$ with $2m/n\to\mu\in(0,\infty)$]
with i.i.d. rate $1$ exponential weights on its edges. Then we have
%
\begin{equation}
\frac{\diam_w(G_n)}{\log n} \stackrel{p} {\rightarrow} \frac{1}{\mu
-1}+
\frac{2}{1 - \lambda_*}.
\end{equation}
\end{theorem}

This result improves on a lower bound of the weighted
diameter given by Bhamidi, van der Hofstad and Hooghiemstra in
\cite{BHH10}, Theorem~2.6.
Note that \cite{BHH10} also deals with the case $np\to\infty$ which is
out of the scope of the present paper.

\subsection{Overview of the proof and organization of the paper}

Our work is a direct generalization of \cite{ding09} with
significantly more involved calculations.
The first key idea of the proof from \cite{ding09} is to grow balls
centered at all vertices
of the graph simultaneously. The time when two balls centered at $a$
and $b$,
respectively, intersect is exactly the half of the weighted distance between
$a$ and $b$. (In what follows, we will sometimes deliberately use the term
\textit{time} instead of the term \textit{weighted distance}.) Hence
the weighted diameter
becomes twice the time when the last two balls intersect. A simple
argument shows that any two balls
containing slightly more than $\sqrt{n}$ vertices ($2\sqrt{rn\log n}$
vertices for $r$-regular case) will intersect with high probability;
see Proposition~\ref{prop-up}.
Hence it will be enough to control the time at which all balls have
reached this critical size of order $\sqrt{n}$ in order to prove an
upper bound for the weighted diameter.
For a proof of the upper bound on the diameter, we apply an union
bound argument as in \cite{ding09}. Hence, we need to
find the right time such that the probability for a (typical) ball to
reach size $\sqrt{n}$ is of order~$n^{-1}$.
In order to do so, we use the second main idea of the proof: we couple
the exploration process on the weighted graph with a continuous time
Markov branching process.
This coupling argument is quite standard, and we will deal here with
the same branching process
approximation for the exploration process on the graph as in
\cite{BHH09}. However, we are facing here new difficulties as we need
to consider events here of small probability for this exploration
process (of order $n^{-1}$). In particular, we need to show that the
coupling is still valid for such large deviations. When $\dmin\geq3$,
the argument of Ding et al. \cite{ding09} can be extended easily \cite
{amdrle-fl}. But as soon as
$\dmin\leq2$, several complications happen. First as shown in
\cite{amlel-12}, the asymptotics for the large deviations of the
branching process depend on the minimal possible offspring. Second, as
soon as $\dmin=1$, the small components of the graph contain now a
positive fraction of the nodes. We need to bound the diameter of these
small components and to study the diameter on the largest component,
we need to condition our exploration process on ``nonextinction.''
Similarly, the presence of degree one nodes significantly complicates
the proof of the lower bound. In order to apply the second moment
method as in \cite{ding09}, we need to first remove vertices with
degree one iteratively to work with the $2$-core of the graph (indeed
an augmented version of this $2$-core; see Section~\ref{sec:low1c} for
details).


We consider in Section~\ref{sec:fpp-proof} the exploration process for
configuration model which consists in growing balls simultaneously from
each vertex. 
A precise treatment of the exploration process, resulting in
information about the growth rates of the balls, is given in this
section. In addition, the section provides some necessary notation and
definitions that will be used throughout the last three sections.
Sections~\ref{sec:fpp-upper} and~\ref{sec:fpp-lower} form the heart of
the proof. We first prove that the above bound is an upper bound for
the weighted diameter. This will consist of defining the two parameters
$\alpha_n$ and $\beta_n$ with the following significance: (i) Two
balls of size at least $\beta_n$ intersect almost surely; (ii)
considering the growing balls centered at a vertex in the graph, the
time it takes for the balls to go from size $\alpha_n$ to size $\beta
_n$ have all the same asymptotic for all the vertices of the graph, and
the asymptotic is half of the typical weighted distance in the graph;
(iii) the time it takes for the growing balls centered at a given
vertex to reach size at least $\alpha_n$ is upper bounded by $\frac
{1+\varepsilon}{\Gamma(\dmin)}\log n$ for all $\varepsilon>0$ with high
probability (w.h.p.). This will show that the diameter is w.h.p.
bounded above by $(1+\varepsilon)(\frac{1}{\nu-1}+\frac{2} {\Gamma(\dmin
)})\log n$,\vspace*{1.5pt} for all $\varepsilon>0$. The last section provides the
corresponding lower bound. To obtain the lower bound, we show that
w.h.p. (iv)~there are at least two nodes with degree $\dmin$ such
that the time it takes for the balls centered at these vertices to
achieve size at least $\alpha_n$ is worse than the other vertices, and
is lower bounded by $\frac{1-\varepsilon}{\Gamma(\dmin)} \log n$, for all
$\varepsilon>0$. And using this, we conclude that the diameter is w.h.p.
bounded below by $(1-\varepsilon)(\frac{1}{\nu-1} + \frac{2}{\Gamma(\dmin
)})\log n$, for all fixed $\varepsilon>0$, finishing\vspace*{1.5pt} the proof of our main theorem.

The actual values of $\alpha_n$ and $\beta_n$ will be
%
\begin{equation}
\label{eq:defab}\alpha_n:= \bigl\lfloor\log^3 n\bigr
\rfloor\quad\mbox{and} \quad\beta _n:= \biggl\lfloor3\sqrt{
\frac{\mu}{\nu-1} n \log n}\biggr\rfloor.
\end{equation}

When $\dmin=1$, the longest shortest path in a random graph will be
between a pair of vertices $a$ and $b$ of degree one. Furthermore, this
path consists of a path from $a$ to the $2$-core, a path through the
$2$-core and a path from the $2$-core to $b$. For this, we need to
provide some preliminary results on the structure of the 2-core; this
is done in Appendix \ref{sec:fpp-2core}.
In Appendix \ref{sec:Sko}, we show that our results still apply for
random graphs $G(n,p)$ and $G(n,m)$ by conditioning on the degree sequence.

\subsubsection*{Basic notation}
We usually do not make explicit reference to the probability space
since it is usually clear to which one we are referring. We say that an
event $A$ holds almost surely, and we write a.s., if $\PP(A) = 1$. The
indicator function of an event $A$ is of particular interest, and it is
denoted by $\ind[A]$.
We consider the asymptotic case when $n \rightarrow\infty$, and say
that an event holds w.h.p. (with
high probability) if it holds with probability tending to 1 as $n
\rightarrow\infty$. We denote by $\tod$ and $\top$ convergence in
distribution, and in probability, respectively. Similarly, we use $o_p$
and $O_p$ in a standard way. For example, if $(X_n)$ is a sequence of
random variables, then $X_n = O_p(1)$ means that ``$X_n$ is bounded in
probability,'' and $X_n=o_p(n)$ means that $X_n/n \stackrel
{p}{\rightarrow} 0$.

\section{First passage percolation in $G^*(n, (d_i)_1^n)$}\label{sec:fpp-proof}

We start this section by introducing some new notation and
definitions. Before this, one remark is in order. In what follows, we
will sometimes deliberately use the term \textit{time} instead of the
term \textit{weighted distance}. It will be clear from the context what
we actually mean by this.

Let $ (G=(V,E),w )$ be a weighted graph. For a vertex $a\in V$
and a real number $t>0$, the $t$-radius neighborhood of $a$ in the
(weighted) graph, or the ball of radius~$t$ centered at $a$, is defined as
\[
B_w(a,t):= \bigl\{ b, \dist_w(a,b)\leq t \bigr\}.
\]
The first time $t$ where the ball $B_w(a,t)$ reaches size $k+1$ will
be denoted by $T_a(k)$ for $k\geq0$, that is,
\[
T_a(k) = \min \bigl\{ t\dvtx \bigl|B_w(a,t)\bigr| \geq k+1
\bigr\},\qquad T_a(0) = 0.
\]
If there is no such $t$, that is, if the component containing $a$ has
size at most $k$, we define $T_a(k) = \infty$. More precisely, we use
$I_a$ to denote the size of the component containing $a$ in the graph
minus one. In other words,
\[
I_a:= \max \bigl\{\bigl |B_w(a,t)\bigr|, t\geq0 \bigr\}-1,
\]
so that for all $k>I_a$, we set $T_a(k)=\infty$.
Note that there is a vertex in $B_w(a,T_a(k))$ which is not in any ball
of smaller radius around $a$. When the weights are i.i.d. according to
a random variable with continuous density, this vertex is, in addition,
unique with probability one. We will assume this in what follows.
For an integer $i\leq I_a$, we use $\hd_a(i)$ to denote the
forward-degree of the (unique) node added at time $T_a(i)$ in
$B_w(a,T_a(i))$. Recall that the forward-degree is the degree minus one.
Define $\hS_a(i)$ as follows:
%
\begin{equation}
\label{eq:defhS}\hS_a(i):= d_a +
\hd_a(1) + \cdots + \hd_a(i) - i,\qquad \widehat
S_a(0) =d_a.
\end{equation}
For a connected graph $H$, the tree excess of
$H$ is denoted by $\tx(H)$, which is the maximum number of edges that
can be deleted from $H$ while still keeping it connected. By an abuse
of notation, for a subset $W \subseteq V$, we denote by $\tx(W)$ the
tree excess of the induced subgraph $G[W]$ of $G$ on $W$. (If $G[W]$ is
not connected, then $\tx(W):=\infty$.)
Consider the growing balls $B_w(a,T_a(i))$ for $0\leq i\leq I_a$
centered at~$a$, and let $X_a(i)$ be the tree excess of $B_w(a,T_a(i))$,
\[
X_a(i): = \tx\bigl( B_w\bigl(a,T_a(i)
\bigr) \bigr).
\]
We extend the definition of $X_a$ to all the integer values by setting
$X_a(i) = X_a(I_a)$ for all $i > I_a$.

The number of edges crossing the boundary of the ball $B_w(a,T_a(i))$
is denoted by $S_a(i)$. A simple calculation shows that
%
\begin{equation}
\label{eq:S(a)}S_a(i) = \hS_a(i) - 2
X_a(i).
\end{equation}
%

We now consider a random graph $G(n, (d_i)_1^n)$ with i.i.d. rate one
exponential weights on its edges, such that the degree sequence
$(d_i)_1^n$ satisfies Condition~\ref{cond-dil}. We let $m^{(n)}$ be the
total degree defined by
$m^{(n)}=\sum_{i=1}^n d_i = \sum_{k\geq0} ku_k^{(n)}$.\vspace*{2pt}

One particularly useful property of the configuration model is that it
allows one to construct the graph gradually, exposing the edges of the
perfect matching, one at a time. This way, each additional edge is
uniformly distributed among all possible edges on the remaining
(unmatched) half-edges. We have the following useful lemma.

\begin{lemma}\label{lem:upX}
For any $k\leq\frac{m^{(n)}-n}{2}$, we have
\[
\PP \bigl(2X_a(k)\geq x\mid\hS_a(k),
I_a\geq k \bigr)\leq\PP \bigl( \Bin \bigl(\hS_a(k),
\sqrt{\hS_a(k)/n} \bigr)\geq x \mid\hS_a(k) \bigr).
\]
\end{lemma}
%

\begin{pf}
To prove this, we need the following intermediate result
proved in~\cite{fern07}, Lemma~3.2. 

\begin{lemma}\label{lem:fpp-excess}
Let $A$ be a set of $m$ points, that is, $|A|=m$, and let $F$ be a
uniform random matching of elements of $A$. For $e \in A$, we denote by
$F(e)$ the point matched to $e$, and similarly for $X \subset A$, we
write $F(X)$ for the set of points matched to $X$. Now let $X \subset
A$, $k=|X|$, and assume $k \leq m/2$. We have
\[
\bigl|X \cap F(X)\bigr| \leq_{\mathrm{st}} \Bin(k,\sqrt{k/m}).
\]
\end{lemma}

Conditioning on all the possible degree sequences $\hd_a(1), \hd
_a(2),\ldots,\hd_a(k)$, with the property that
$d_a+\sum_{1\leq i\leq k} \hd_a(i) = \hS_a(k)$, the configuration model
becomes equivalent to the following process: start from $a$, and at
each step $1\leq i \leq k$, choose a vertex $a_i$ of degree $\hd
_a(i)+1$ uniformly at random from all the possible vertices of this
degree outside the set $\{a,a_1,\ldots,a_{i-1}\}$, choose a half-edge
adjacent to $a_i$ uniformly at random and match it with a uniformly
chosen half-edge from the yet-unmatched half-edges adjacent to one of
the nodes $a,a_1,\ldots,a_{i-1}$. And at the end, after $a_k$ has been
chosen, take a uniform matching for all the remaining $(m^{(n)}-2k)$
half-edges. Now the proof follows from Lemma~\ref{lem:fpp-excess} by
the simple observation that, since $m^{(n)} -2k \geq n$,
\begin{eqnarray*}
&&\PP \bigl( \Bin \bigl(\hS_a(k),\sqrt{\hS_a(k)/m^{(n)}-2k}
\bigr)\geq x \mid\hS _a(k) \bigr)
\\
&&\qquad\leq \PP \bigl( \Bin \bigl(\hS_a(k),\sqrt{\hS_a(k)/n}
\bigr)\geq x \mid\hS_a(k) \bigr).
\end{eqnarray*}
\upqed\end{pf}

In the sequel, we will also need to consider the number of vertices of
forward-degree at least two in the (growing) balls centered at a vertex
$a\in V$. Thus, for $i\leq
I_a$, define
%
\begin{equation}
\label{eq:defg}\gamma_a(i):= \sum_{\ell=1}^i
\ind\bigl[\hd_a(\ell)\geq2\bigr] =\bigl | \bigl\{b \in B_w
\bigl(a,T_a(i)\bigr)\dvtx b\neq a \mbox{ and } d_b
\geq3 \bigr\}\bigr|,
\end{equation}
and extend the definition to all integers by setting
$\gamma_a(i) = \gamma_a(I_a)$ for all $i> I_a$. Note that $\gamma_a(0)
=0$ and $\gamma_a(i)=i$ if $\dmin\geq3$.

Now define $\bT_a(k)$ to be the first time where the ball centered at
$a$ has at least $k$ nodes of forward-degree at least two. More precisely,
%
\begin{equation}
\label{eq:defbT}\bT_a(i):=\min \bigl\{T_a(\ell),
\mbox{for } \ell \mbox{ such that } \gamma_a(\ell)\geq k \bigr\}.
\end{equation}

The main idea of the proof of Theorem~\ref{thm-fpp-main} consists of
growing the balls around each vertex of the graph simultaneously so
that the diameter becomes equal to twice the time when the last two
balls intersect. In what follows, instead of taking a graph at
random and then analyzing the balls, we use a standard
coupling argument in random graph theory which allows us to build the
balls and the graph at the same time. We present this coupling in the
next coming section.

\subsection{The exploration process}\label{sec:explor}

Fix a vertex $a$ in $G^*(n,(d_i)_1^n)$, and consider the following
continuous-time exploration
process. At time $t=0$, we have a neighborhood consisting only of $a$, and
for $t>0$, the neighborhood is precisely $B_w(a,t)$. We now give an
equivalent description of this process. This provides a more
convenient way for analyzing the random variables which are crucial in
our argument, for example, $S_a(k)$. The idea is that instead of taking a
graph at random and then analyzing the balls, the graph and the balls
are built at the same time. We will consider a growing set of vertices
denoted by $B$ and a list $L$ of yet unmatched half-edges in $B$.
Recall that in the usual way of constructing a random graph with given
degree sequence, we match half-edges amongst themselves uniformly at
random. In the following, by a \textit{matching}, we mean a pair of
matched half-edges.

\begin{itemize}
\item Start with $B = \{a\}$, where $a$ has $d_a$ half-edges. For each
half edge, decide (at random depending on the previous choices) if the
half-edge is matched to a half-edge adjacent to $a$ or not. Reveal the
matchings consisting of those half-edges adjacent to $a$ which are
connected amongst themselves (creating self-loops at~$a$) and assign
weights independently at random to these edges. The remaining unmatched
half-edges adjacent to $a$ are stored in a list $L$. (See the next step
including a more precise description of this first step.)

\item Repeat the following exploration step as long as the list $L$ is
not empty.
\item Given there are $\ell\geq1$ half-edges in the current list, say
$L=(h_1,\ldots, h_\ell)$, let $\Psi\sim\operatorname{Exp}(\ell)$ be an
exponential variable with mean $\ell^{-1}$. After time $\Psi$ select
a half-edge from $L$ uniformly at random, say $h_i$. Remove $h_i$
from $L$ and match it to a
uniformly chosen half-edge in the entire graph excluding $L$, say $h$. Add
the new vertex (connected to $h$) to $B$ and reveal the
matchings (and weights) of any of its half-edges whose matched
half-edge is
also in $B$. More precisely, let $d$ be the degree of this new
vertex and $2x$ the number of already matched half-edges in $B$ (including
the matched half-edges $h_i$ and $h$). There is a
total of $m-2x$ unmatched half-edges, $m$ being the total number of
half-edges of the random graph $G$. Consider one of the
$d-1$ half-edges of the new vertex (excluding $h$ which is connected
to $h_i$); with probability $(\ell-1)/(m-2x-1)$ it is matched with a
half-edge in $L$, and with the complementary
probability it is matched with an unmatched half-edge outside
$L$. In the first case, match it to a uniformly chosen half-edge of
$L$, and remove the corresponding half-edge from $L$. In the second
case, add it to $L$. We proceed in the similar manner for all the
$d-1$ half-edges of the new vertex.
\end{itemize}

Let $B(a,t)$ and $L(a,t)$ be, respectively, the set of vertices and the
list generated by the above procedure at time $t$, where $a$ is the
initial vertex.
Considering the usual configuration model, and using the memoryless
property of the exponential distribution, we have
$B_w(a,t)=B(a,t)$ for all $t$. To see this, we can continuously grow
the weights of the half-edges $h_1, \ldots, h_{\ell}$ in $L$ until one
of their rate $1$ exponential clocks fire. Since the minimum of $\ell$
i.i.d. exponential variables with rate 1 is exponential with rate
$\ell$, this is the same as choosing uniformly a half-edge $h_i$ after
time $\Psi$ (recall that by our conditioning, these $\ell$ half-edges
do not pair within themselves). Note that the final weight of an edge
is accumulated between the time of arrival of its first half-edge and
the time of its pairing (except edges going back into $B$ whose
weights are revealed immediately). Then the equivalence follows from
the memoryless property of the exponential distribution.

Note that $T_a(i)$ is the time of the $i$th exploration step in the
above continuous-time exploration process.
Assuming $L(a,T_a(i))$ is not empty, at time $T_a(i+1)$, we match a uniformly
chosen half-edge from the set $L(a,T_a(i))$ to a uniformly chosen
half-edge among all other half-edges, excluding those in
$L(a,T_a(i))$. Let $\cF_{t}$ be the $\sigma$-field generated by the
above process until time $t$.
Given $\cF_{T_a(i)}$, $T_a(i+1)-T_a(i)$ is an
exponential random variable with rate $S_a(i)$, given by equation (\ref
{eq:S(a)}), which is equal to $|L(a,T_a(i))|$, the size of the list
consisting of unmatched half-edges in $B(a,T_a(i))$. In other words,
\[
\bigl( T_a(i+1)-T_a(i) | \cF_{T_a(i)} \bigr)
\stackrel{d} {=} \operatorname {Exp}\bigl(S_a(i)\bigr),
\]
this is true since the minimum of $k$ i.i.d. rate one exponential
random variables is an exponential of rate k.

Recall that $I_a=\min\{i, S_a(i)=0\} \leq n-1$, and set $S_a(i)=0$ for
all $I_a\leq i\leq n-1$. We now extend the definition of the sequence
$\hd(i)$ to all the values of $i\leq n-1$, constructing a sequence $(\hd
_a(i))_{i=1}^{n-1}$ which will coincide in the range $i\leq I_a$ with
the sequence $\hd_a(i)$ defined in the previous subsection. We first
note that in the terminology of the exploration process, the sequence
$(\hd_a(i))_{i\leq I_a}$ can be constructed as follows. At time
$T_a(i+1)$, the half-edge adjacent to the $i+1$th vertex is chosen
uniformly at random from the set of all the half-edges adjacent to a
vertex out-side $B$, and $\hd(i+1)$ is the forward-degree of the vertex
adjacent to this half-edge. Thus the sequence $(\hd(i))_{i\leq I_a}$
has the following description.

Initially, associate to all vertices $j$ a set of $d_j$ half-edges
(corresponding the set of half-edges outside $B$ and $L$). At step $0$,
remove the half-edges corresponding to vertex $a$. Subsequently, at
step $k \leq I_a$, choose a half-edge uniformly at
random among all the remaining half-edges; if the half-edge is drawn
from the node $j$'s half-edges,
then set $\hd_a(k) = d_{j} - 1 $, and remove the node $j$ and all of
its half-edges. Obviously, this description allows us to extend the
definition of $\hd_a(i)$ to all the values of $I_a<i \leq n-1$. Indeed,
if $I_a<n-1$, there are still half-edges at step $I_a+1$, and we can
complete the sequence $\hd_a(i)$ for $i\in[I_a+1,n-1]$ by continuing
the sampling described above. In this way,\vspace*{1pt} we obtain a sequence
$(\hd_a(i))_{i=1}^{n-1}$ which coincides with the sequence defined in
the previous section for
$i\leq I_a$.


We also extend the sequence $\hS_a(i)$ for $i> I_a$
thanks to (\ref{eq:defhS}). Recall that we set $X_a(i)=X_{a}(I_a)$ for all
$i>I_a$. It is simple to see that with these conventions, relation
(\ref{eq:S(a)}) is not anymore valid for $i>I_a$, but we still have
$S_a(i)\leq
\hS_a(i)-2X_a(i)$ for all $i$.

The process $i\mapsto X_a(i)$ is nondecreasing in
$i\in[1,n-1]$. Moreover, given $\cF_{T_a(i)}$, the increment
$X_{a}(i+1)-X_a(i)$ is
stochastically dominated by the following binomial random variable:
%
\begin{equation}
\label{upinc1}X_{a}(i+1)-X_a(i) \leq_{\mathrm{st}}
\Bin \biggl( \hd_a(i+1), \frac
{(S_a(i)-1)^+}{m^{(n)}-2(X_a(i)+i)} \biggr),
\end{equation}
where $m^{(n)}=\sum_{i=1}^n d_i $.
We recall here that for two real-valued random variables $A$ and $B$,
we say $A$ is stochastically dominated by $B$ and write $A\leq_{\mathrm{st}} B$
if for all $x$, we have $\PP(A\geq x)\leq\PP(B\geq x)$. If $C$ is
another random variable, we write $A\leq_{\mathrm{st}} (B | C)$ if for all $x$,
$\PP(A\geq x)\leq\PP(B\ge x | C)$.

Note that if $i>I_a$, then $S_a(i)=0$ and $X_{a}(i+1)-X_a(i)=0$, so
that (\ref{upinc1}) is still valid.

For $i< \frac{n}2$, we have
\begin{eqnarray*}
\frac{(S_a(i)-1)^+}{m^{(n)}-2(X_a(i)+i)}&\leq& \frac{\hS
_a(i)-2X_a(i)}{m^{(n)}-2(X_a(i)+i)}
\\
&\leq&\frac{\hS_a(i)}{m^{(n)}-2i} \leq\frac{\max_{\ell\leq i}\hS_a(\ell
)}{n-2i}.
\end{eqnarray*}

We conclude:

\begin{lemma}\label{lem-upinc11}
For $i< \frac{n}2$, we have
%
\begin{equation}
\label{upinc11}X_a(i)\leq_{\mathrm{st}}\Bin \biggl(\max
_{\ell\leq i}\hS_a(\ell)+i, \frac{\max_{\ell\leq i}\hS_a(\ell)}{n-2i} \biggr).
\end{equation}
\end{lemma}

An important ingredient in the proof will be the coupling of the
forward-degree sequence $\{\hd(i)\}$ to an i.i.d. sequence in the range
$i \leq\beta_n$, that we provide in the next subsection.

Recall that we defined $\alpha_n$ and $\beta_n$ as follows [cf.
equation (\ref{eq:defab})]:
\[
\alpha_n = \bigl\lfloor\log^3 n\bigr\rfloor\quad\mbox{and}\quad
\beta_n = \biggl\lfloor3\sqrt {\frac{\mu}{\nu-1} n \log n}\biggr
\rfloor.
\]

\subsection{Coupling the forward-degrees sequence \texorpdfstring{$\widehat{d}_a(i)$}{$d_a(i)$}} \label{sse:fpp-coupling}

We now present a coupling of the variables $\{\hd_a(1), \ldots, \hd_a(k)
\}$
valid for $k\leq\beta_n$, where $\beta_n$ is defined in equation
(\ref{eq:defab}), with an i.i.d. sequence of random variables, that we
now define. Let $\Delta_n:= \max_{i\in[1,n]} d_i$. Note that by
Condition \ref{cond-dil}(ii), we have $\Delta_n = O(n^{1/2-\varepsilon})$.

Denote the order statistics of the sequence of degrees $(d_i^{(n)})$ by
%
\begin{equation}
\label{in:order} d^{(n)}_{(1)}\leq d^{(n)}_{(2)}
\leq\cdots\leq d^{(n)}_{(n)}.
\end{equation}
Define $\um^{(n)}:=\sum_{i=1}^{n-\beta_n} d^{(n)}_{(i)}$, and let $\upi
^{(n)}$ be
the size-biased empirical distribution with the $\beta_n$ highest
degrees in (\ref{in:order}) removed, that is,
\[
\upi^{(n)}_k:= \frac{\sum_{i=1}^{n-\beta_n} (k+1) \ind
[d^{(n)}_{(i)}=k+1 ]}{\um^{(n)}}.
\]
Similarly, define $\bm^{(n)}:=\sum_{i=(\beta_n+1)\Delta_n}^{n}
d^{(n)}_{(i)}$, and let $\bpi^{(n)}$ be the size-biased empirical distribution
with the $(\beta_n+1)\Delta_n$ lowest degrees in (\ref{in:order})
removed, that is,
\[
\bpi^{(n)}_k:= \frac{\sum_{i=(\beta_n+1)\Delta_n}^{n} (k+1) \ind
[d^{(n)}_{(i)}=k+1 ]}{\bm^{(n)}}.
\]
Note that by Condition \ref{cond-dil}, we have
$\beta_n\Delta_n=o(n)$ which implies that both the distributions $\upi^{(n)}$
and $\bpi^{(n)}$ converge to the size-biased distribution $q$ defined
in equation~(\ref{eq:defq}) as $n$ tends to infinity.

The following basic lemma, proved in \cite{amdrle-fl}, Lemma~4.1, shows
that the forward-degree of the $i$th vertex given the forward-degrees
of all the previous vertices is stochastically between two random
variables with lower and upper distributions $\bpi^{(n)}$ and $\upi
^{(n)}$ defined above, provided that $ i\leq\beta_n$. More precisely:


\begin{lemma}\label{lem-coupl}
For a uniformly chosen vertex $a$, we have for all $i \leq\beta_n$,
%
\begin{equation}
\label{eq:coupl}\uD^{(n)}_i \leq_{\mathrm{st}} \bigl(
\hd_a(i) | \hd_a(1),\ldots, \hd_a(i-1)
\bigr) \leq_{\mathrm{st}} \bD^{(n)}_i,
\end{equation}
where $\uD^{(n)}_i$ (resp., $\bD^{(n)}_i$) are i.i.d. with distribution
$\upi^{(n)}$ (resp., $\bpi^{(n)}$).

In particular, we have for all $i \leq\beta_n$,
\[
\sum_{k=1}^{i} \uD^{(n)}_k
\leq_{\mathrm{st}} \sum_{k=1}^{i}
\hd_a(k) \leq_{\mathrm{st}} \sum_{k=1}^{i}
\bD^{(n)}_k.
\]
\end{lemma}

\section{Proof of the upper bound}\label{sec:fpp-upper}
In this section we present the proof of the upper bound for Theorem~\ref
{thm-fpp-main}. Namely we prove that for any $\varepsilon> 0$, with high
probability for all vertices $u$ and $v$ which are in the same
component [i.e., such that $\dist_w(u,v) < \infty$], we have
\[
\dist_w(u,v) \leq \biggl(\frac{1}{\nu-1} + \frac{2}{\Gamma(\dmin)}
\biggr)
(1+\varepsilon) \log n,
\]
where $\Gamma(\dmin)$ is defined in (\ref{def:Gamma}).


The proof will be based on the following two technical propositions.
For the sake of readability, we postpone the proof of these two
propositions to the end of this section.

The first one roughly says that for all~$u$ and $v$, the growing balls
centered at $u$ and $v$ intersect w.h.p. provided that they contain
each at least $\beta_n$ nodes. More precisely:

\begin{proposition}\label{prop-up}
We have w.h.p.
\[
\dist_w(u,v) \leq T_u(\beta_n) +
T_{v}(\beta_n)\qquad \mbox{for all $u$ and $v$}.
\]
\end{proposition}

The above proposition shows that in proving the upper bound, it will
be enough to control the random variable $T_u(\beta_n)$ for each node
$u$ in $V$. 
It turns out that in the range between $\alpha_n$ and $\beta_n$, in the
cases $\dmin\geq3$, $\dmin=2$ and $\dmin=1$, $T_u(k)$ have more or
less the same behavior; namely, it takes time at most roughly half of
the typical (weighted) distance to go from size $\alpha_n$ to $\beta
_n$. More precisely:

\begin{proposition}\label{lem-up-rq}
For a uniformly chosen vertex $u$ and any $\varepsilon>0$, we have
\[
\PP \biggl( T_u(\beta_n) - T_u(
\alpha_n) \geq\frac{(1+\varepsilon)\log
n}{2(\nu-1)} \Big\mid I_u \geq
\alpha_n \biggr) = o\bigl(n^{-1}\bigr).
\]
\end{proposition}

The conditioning $I_u \geq\alpha_n$ is here to ensure that the
connected component which contains $u$ has size at least
$\alpha_n$. In particular, note that one immediate corollary of the
two above propositions is that two nodes whose connected components
have size at least $\alpha_n$ are in the same component (necessarily
the giant component), and that the two balls of size $\beta_n$
centered at these two vertices intersect w.h.p.

Using the above two propositions, we are only left to understand
$T_u(\alpha_n)$, and for this we will need to consider the cases $\dmin
\geq2$ and $\dmin=1$ separately. Before going through the proof of the
upper bound in these cases, we need one more result. Consider the
exploration process started at a vertex $a$. We will need to find lower
bounds for $S_a(k)$ in the range $1\leq k\leq\alpha_n$. Recall that we
defined $\gamma_a(k)$ as the number of nodes of forward-degree at least
two in the growing balls centered at~$a$; cf. equation (\ref{eq:defg})
for the precise definition. These nodes are roughly all the ones which
could contribute to the growth of the random variable $S_a(k)$. Now
define the two following events:
\begin{eqnarray*}
R_a &:=& \bigl\{ S_a(k) \geq\dmin+
\gamma_a(k),\mbox{ for all } 0 \leq k \leq\alpha_n-1
\bigr\},
\\
R'_a &:=& \bigl\{ S_a(k) \geq
\gamma_a(k),\mbox{ for all } 0 \leq k \leq\alpha_n-1
\bigr\}.
\end{eqnarray*}

\begin{lemma}\label{lem-R1}
Assume $d_a \geq2$ and $\hd_a(i)\geq1$ for all { $1\leq i \leq\alpha
_n$}. Then we have
%
\begin{eqnarray}
\PP \bigl( R_a\mid\hd_a(1),\ldots,
\hd_a(n-1) \bigr) &\geq& 1-o\bigl(\log^{10} n / n\bigr),
\label{pr-R1}
\\
\PP \bigl( R'_a\mid\hd_a(1),\ldots,
\hd_a(n-1) \bigr) &\geq& 1 - o\bigl(n^{-3/2}\bigr).
\label{pr-R2}
\end{eqnarray}
In particular, $\PP(R_a)\geq1-o(\log^{10} n / n)$ and $\PP(R'_a)\geq
1 - o(n^{-3/2})$.
\end{lemma}

\begin{pf}
Since $\hd_a(i)\geq1$, $\hS_a(k)$ is nondecreasing in $k$. We have
for all $k\leq\alpha_n$,
%
\begin{equation}
\label{eq:ineqf} \dmin+ \gamma_a(k) \leq d_a +
\gamma_a(k) \leq\hS _a(k) \leq \alpha_n
\Delta_n = o(n),
\end{equation}
and moreover, $\max_{k\leq\alpha_n}\hS_a(k) = \hS_a(\alpha_n)$.
Since $d_a\geq2$ and $S_a(k) = \widehat S_a(k) - 2 X_a(k)$, we have
\[
\bigl\{X_a(\alpha_n)=0\bigr\}\subset
R_a, \qquad\bigl\{X_a(\alpha_n)\leq1\bigr\}
\subset R'_a.
\]
Note that the inequalities in (\ref{eq:ineqf}) are true for any
sequence such that $1\leq\hd_a(i)\leq
\Delta_n$. In particular, in the rest of the proof we condition on a
realization of the sequence $\bold{d}=(d_a,\hd_a(1),\ldots,\hd_a(n-1))$.

We distinguish two cases depending on whether or not $\hS_a(\alpha_n)$
is smaller than~$3 \alpha_n$. Denote this event by $\mathcal Q$ (and
its complementary by $\mathcal Q^c$), that is,
\[
\mathcal Q:= \bigl\{ \hS_a(\alpha_n) < 3
\alpha_n \bigr\}.
\]
\begin{itemize}
\item\textit{Case} (1). $\hS_a(\alpha_n) < 3 \alpha_n$.
Conditioning on $\mathcal Q$, by Lemma~\ref{lem-upinc11} we have
\[
X_a(\alpha_n) \leq_{\mathrm{st}} \Bin \biggl( 4
\alpha_n, \frac{3 \alpha
_n}{n-2\alpha_n} \biggr).
\]
Thus we have
\begin{eqnarray*}
\PP \bigl( X_a(\alpha_n) \geq1 \mid\mathcal Q,\bold{d}
\bigr) \leq \PP \biggl( \Bin \biggl(4 \alpha_n, \frac{3\alpha_n}
{n-2\alpha_n}
\biggr) \geq1 \biggr) &\leq& O\bigl(\alpha_n^2/n\bigr),
\\
\PP \bigl( X_a(\alpha_n) \geq2 \mid\mathcal Q,\bold{d}
\bigr) \leq \PP \biggl( \Bin \biggl(4 \alpha_n,
\frac{3 \alpha_n}{n-2\alpha_n}
\biggr) \geq2 \biggr) &\leq&O\bigl(\alpha_n^4/n^2
\bigr).
\end{eqnarray*}
We infer that
\begin{eqnarray*}
\PP \bigl( (R_a)^c \mid\mathcal Q,\bold{d} \bigr) &
\leq& O\bigl(\alpha _n^2/n\bigr),
\\
\PP \bigl( \bigl(R'_a\bigr)^c \mid
\mathcal Q,\bold{d} \bigr) &\leq& O\bigl(\alpha_n^4/n^2
\bigr).
\end{eqnarray*}

\item\textit{Case} (2). $\hS_a(\alpha_n) \geq3
\alpha_n$. Note that in this case, we still have
\[
\max_{k\leq\alpha_n}\hS_a(k) = \hS_a(
\alpha_n)\leq \alpha_n\Delta_n = o(n).
\]
Moreover, there exists $k\leq\alpha_n$ such that for all $\ell\leq k$,
$\hS_a(\ell)< 3\alpha_n$ and $\hS_a(k+1)\geq3\alpha_n $. {Note that
since we have conditioned on the degree sequence $\bold{d}$, the value
of $k$ is deterministic ($k$ is not a random variable).} Conditioning
on the event $\mathcal Q^c$, we obtain by Lemma~\ref{lem-upinc11},
%
\begin{eqnarray}
\label{in:kmax} X_a(k)&\leq_{\mathrm{st}}& \Bin \biggl( 4
\alpha_n,\frac{3\alpha_n}{n-2\alpha_n} \biggr) \quad\mbox{and}
\nonumber
\\[-8pt]
\\[-8pt]
\nonumber
X_a(\alpha_n) &\leq_{\mathrm{st}}& \Bin
\biggl(\alpha_n(\Delta_n+1), \frac{\alpha_n\Delta_n}{n-2\alpha_n} \biggr).
\end{eqnarray}

By Condition~\ref{cond-dil}(ii), there exists a $\varepsilon>0$ such
that $\Delta_n:= O(n^{1/2 -\varepsilon})$. Let $m = \lceil2 \varepsilon
^{-1}\rceil$. Combining the last (stochastic) inequality together with
the Chernoff's inequality applied to the right-hand side binomial
random variable, we obtain
\begin{eqnarray*}
\PP \bigl( X_a(\alpha_n) \geq m \mid\mathcal
Q^c,\bold{d} \bigr) &\leq & \PP \biggl(\Bin \biggl(
\alpha_n(\Delta_n+1), \frac{\alpha_n\Delta
_n}{n-2\alpha_n} \biggr) \geq m
\biggr)
\\
&=& O \bigl( \bigl(\Delta^2_n \alpha^2_n
/ n\bigr)^m \bigr) = o\bigl(n^{-3}\bigr).
\end{eqnarray*}
We notice that for all $\ell> k$, we have $S_a(\ell)\geq
2\alpha_n-2X_a(\alpha_n)$. Also for $n$ large enough, we have
$2\alpha_n-2m\geq\dmin+\gamma_a(\ell)$.
Therefore,
\begin{eqnarray*}
\bigl\{ X_a(k)=0, X_a(\alpha_n)\leq m,
\mathcal Q^c \bigr\} &\subset& R_a\cap \mathcal
Q^c \quad\mbox{and}
\\
\bigl\{ X_a(k)\leq1, X_a(\alpha_n)\leq
m, Q_1^c \bigr\} &\subset& R'_a
\cap\mathcal Q^c.
\end{eqnarray*}
This in turn implies that
\begin{eqnarray*}
\PP \bigl( (R_a)^c\mid\mathcal Q^c,
\bold{d} \bigr) &\leq& \PP \bigl( X_a(k)\geq1\mid\mathcal
Q^c \bigr)+\PP \bigl( X_a(\alpha_n) \geq m
\mid\mathcal Q^c \bigr)
\\
&\leq& O\bigl(\alpha_n^2/n\bigr)\quad \mbox{and}
\\
\PP \bigl( \bigl(R'_a\bigr)^c\mid
\mathcal Q^c,\bold{d} \bigr) &\leq& \PP \bigl( X_a(k)
\geq2\mid \mathcal Q^c \bigr)+ \PP \bigl( X_a(
\alpha_n) \geq m\mid\mathcal Q^c \bigr)
\\
&\leq& O\bigl(\alpha_n^4/n^2\bigr).
\end{eqnarray*}
In the above inequalities, we used (stochastic) inequality~(\ref
{in:kmax}) and case (1) to bound the terms $\PP ( X_a(k)\geq1\mid
\mathcal Q^c  )$ and $\PP ( X_a(k)\geq2\mid
\mathcal Q^c  )$.
\end{itemize}
The lemma follows by the definition of $\alpha_n$.
\end{pf}

We are now in position to provide the proof of the upper bound in the
different cases depending on whether $\dmin\geq2$ or $\dmin=1$.


\textit{In what follows, we will use the following property of the exponential
random variables, without sometimes mentioning: If $Y$ is an
exponential random
variable of rate $\mu$, then for any $\theta<\mu$, we have
$\EE [ e^{\theta Y} ] = \frac{\mu}{\mu-\theta}$.}

\subsection{Proof of the upper bound in the case \texorpdfstring{$\dmin\geq2$}{$d_{\min}>=2$}}

Consider the exploration process defined in Section~\ref{sec:explor}
starting from $a$.
Recall definitions (\ref{eq:defg}) and (\ref{eq:defbT}):
$\gamma_a(i)$ is the number of nodes with forward-degree (strictly)
larger than one until the $i$th exploration step, and $\bT_a(k)$ is the
first time
that the $k$th node with the forward-degree (strictly) larger than one appears
in the exploration process started at node~$a$.
We also define the sets
\[
L_a(k):= \bigl\{ \ell, \bT_a(k)\leq
T_a(\ell) <\bT_a(k+1) \bigr\},
\]
for $k\geq0$, and let $n_a(k)$ be the smallest $\ell$ in $L_a(k)$.
Clearly, we have $n_a(k)\geq k$ and
\[
\gamma_a^{-1}(k) = L_a(k) =
\bigl[n_a(k),n_a(k+1)-1\bigr].
\]

 Note that in the case $\dmin\geq3$, we have $q_1=\bpi
_1^{(n)}=\upi_1^{(n)}=0, \gamma_a(k) = k$, $\bT_a(k) = T_a(k)$ and
$L_a(k) = \{k\}$. This case is also treated in \cite{amdrle-fl}.
However, our arguments bellow are still valid in this case.

For $x, y \in\mathbb{R}$, we denote $x \wedge y = \min(x,y)$. We will
need the following lemma.

\begin{lemma}\label{lem-up-bis2}
For a uniformly chosen vertex $a$, any $x>0$ and any $\ell=O(\log n)$,
we have
\[
\PP \bigl( T_a(\alpha_n \wedge I_a)\geq
x \log n + \ell \bigr) \leq o\bigl(n^{-1}\bigr) + o
\bigl(e^{-\dmin(1-q_1)\ell}\bigr).
\]
\end{lemma}

\begin{pf}
Recall that given the sequence $S_a(k)$, for $k<I_a$, the random
variables $T_a(k+1)-T_a(k)$ are i.i.d. exponential random variables
with mean $S_a(k)^{-1}$.
First write
\begin{eqnarray*}
T_a(\alpha_n ) &=& \sum_{0\leq j < \alpha_n }T_a(j+1)-T_a(j)
\\
&\leq& \sum_{k\leq K_n}\bT_a(k+1) -
\bT_a(k),
\end{eqnarray*}
where $K_n$ is the largest integer such that $n_a(K_n)\leq\alpha_n$.

We now show that for any $x>0$ and $\ell=O(\log n)$,
%
\begin{equation}
\label{eq:Ra}\PP \bigl( T_a(\alpha_n)\geq x \log n +
\ell, R_a \bigr) = o\bigl(e^{-\dmin(1-q_1)\ell}\bigr).
\end{equation}

Note that a sum of a geometric (with parameter $\pi$) number of
independent exponential random variables with parameter $\mu$ is
distributed as an exponential random variable with parameter
$(1-\pi)\mu$.
For any $k\leq K_n$, we have
\[
\bT_a(k+1) - \bT_a(k) = \sum
_{j\in L_a(k)}T_a(j+1)-T_a(j).
\]
Assume $R_a$ holds. Then we have $S_a(j)\geq\dmin+k$ for all $j\in
[n_a(k),n_a(k+1)-1]=L_a(k)$. { Thus
\[
T_a(j+1)-T_a(j) \leq_{\mathrm{st}}
Y_{k,i} \sim\operatorname{Exp}(\dmin+k),
\]
where $i = j - n_a(k) + 1$, and all the $Y_{k,i}$'s are independent.
[For $i =1, \ldots,\break |L_a(k)|$, $Y_{k,i}$ are exponential random
variables with rate $\dmin+k$.]}

For any positive $t$ and $\theta$, we obtain [for $\mathbf{d}_a:= (d_a,
\hd_a(1), \ldots, \hd_a(n-1))$]
\begin{eqnarray*}
\PP \bigl( T_a(\alpha_n) - \bT_a(1)\geq
t, R_a \bigr) &\leq& \EE \biggl[ \EE \biggl[ \ind(R_a)
\prod_{1 \leq k\leq K_n}e^{\theta(\bT
_a(k+1) -
\bT_a(k))}\Big\mid\mathbf{d}_a
\biggr] \biggr]e^{-\theta t}
\\
&=& \EE \biggl[ \prod_{1 \leq k\leq K_n}e^{\theta\sum
_{i=1}^{|L_a(k)|}Y_{k,i}}\PP (
R_a\mid\mathbf{d}_a ) \biggr]e^{-\theta t}
\\
&\leq&\prod_{1 \leq k\leq
\alpha_n} \biggl(1+\frac{\theta}{(\dmin+k)(1-\upi_1^{(n)})-\theta}
\biggr)e^{-\theta t},
\end{eqnarray*}
where in the last inequality, we used the fact that the
probability for a new node to have forward-degree one is at most
$\upi_1^{(n)}$, and so the length $|L_a(k)|$ is dominated by a
geometric random variable with parameter $\upi_1^{(n)}$.
Taking $\theta= \dmin(1-\upi_1^{(n)})$ in the above inequality, we get
\begin{eqnarray*}
\PP \bigl( T_a(\alpha_n) -\bT_a(1) \geq
t, R_a \bigr) &\leq& \prod_{1 \leq k\leq\alpha_n}
\biggl(1 + \frac{\dmin
(1-\upi_1^{(n)})}{(1-\upi_1^{(n)})k} \biggr)e^{-\dmin(1-\upi_1^{(n)})
t}
\\
&=& \prod_{1 \leq k\leq\alpha_n} (k+\dmin)/k e^{-\dmin(1-\upi
_1^{(n)}) t}
\\
&<& \alpha_n^3e^{-\dmin(1-\upi_1^{(n)}) t}. 
\end{eqnarray*}
In the same way, we can easily deduce that
\[
\bigl( \bT_a(1) \mid R_a \bigr) \leq_{\mathrm{st}}
\operatorname{Exp}\bigl(\dmin \bigl(1-\upi_1^{(n)}\bigr)\bigr).
\]

Let $t=x \log n + \ell$, and note that $\ell\leq C\log n$ for some
large constant $C>0$ { [by assumption $\ell= O(\log n)$]}. Take any
$0 < \varepsilon< x(1-q_1)(C+x)^{-1}$; since for $n$ sufficiently large,
we have $\upi_1^{(n)} \leq q_1+\varepsilon$, we obtain
\[
\PP \bigl( T_a(\alpha_n)\geq x\log n +\ell,
R_a \bigr) \leq \frac{\alpha_n^3}{n^{\dmin(x(1-q_1-\varepsilon)-\varepsilon C)}}e^{-\dmin
(1-q_1)\ell},
\]
and (\ref{eq:Ra}) follows. Note that $x(1-q_1-\varepsilon)-\varepsilon C >0$
by the choice of $\varepsilon$.

 Assume now that the event $R'_a\cap R_a^c$ holds. Two cases
can happen: either $I_a<\alpha_n$ or $I_a \geq\alpha_n$.

[Note that in the case $\dmin\geq3$, by Lemma~\ref{lem-R1}
we have $\PP(I_a \geq\alpha_n ) \geq1 - o(n^{-3/2})$. Indeed, for
$\dmin\geq3$, we have $\gamma_a(k)=k$ so that
$ R'_a \subseteq\{I_a \geq\alpha_n\} = \{S_a(k) \geq1, \mbox{ for
all $0\leq k \leq\alpha_n-1$}\}$.]

 If $I_a < \alpha_n$, then by the definition of $R'_a$,
$0=S_{a}(I_a)\geq\gamma_{a}(I_a)$, that is, $\gamma_a(I_a) =0$. In other
words, the component
of $a$ is a union of cycles (or loops) having node $a$ as a common
node, and with total number of edges less than $\alpha_n$. Hence, in
this case, we have
\begin{eqnarray*}
&&\PP \bigl( R'_a, R_a^c,
I_a <\alpha_n, T_a(I_a)
\geq x\log n + \ell \bigr)
\\
&&\qquad \leq\PP \bigl( R_a^c \mid\mathbf{d}_a
\bigr) \biggl(\sum_{0\leq k\leq
\alpha_n} \bigl(\upi_1^{(n)}
\bigr)^k\int_{x \log n + \ell}^\infty t^k
\frac
{e^{-t}}{k!}\,dt \biggr)
\\
&&\qquad\leq \log^{10} n /n  \bigl(1-\upi_1^{(n)}
\bigr)^{-1}\exp \bigl(-\bigl(1-\upi _1^{(n)}
\bigr) (x \log n + \ell) \bigr) = o\bigl(n^{-1}\bigr),
\end{eqnarray*}
where the last inequality follows from inequality~(\ref{pr-R1}) in
Lemma~\ref{lem-R1}.

 In the second case, when $I_a\geq\alpha_n$, let
\[
\mathcal Q = R'_a\cap R_a^c
\cap \{ I_a\geq\alpha_n \}.
\]
If $\mathcal Q$ holds, by the definition of $R'_a$, we have $S_a(j)\geq
k$ for all $j \in L_a(k)$. { Thus
\[
T_a(j+1)-T_a(j) \leq_{\mathrm{st}}
Y_{k,i} \sim\operatorname{Exp}(k),
\]
where $i = j - n_a(k) + 1$, and all the $Y_{k,i}$'s are independent.
[For $i =1, \ldots,\break |L_a(k)|$, $Y_{k,i}$ are exponential random
variables with rate $k$.]}
Hence, by the same argument as above, we have
\begin{eqnarray*}
\PP \bigl( T_a(\alpha_n)-\bT_a(2)\geq t,
\mathcal Q \bigr) &\leq& \EE \biggl[ \EE \biggl[ \ind(\mathcal Q)\prod
_{2\leq k\leq K_n}e^{\theta
(\bT_a(k+1) -
\bT_a(k))}\Big\mid\mathbf{d}_a \biggr]
\biggr]e^{-\theta t}
\\
&\leq&\EE \biggl[ \prod_{2\leq k\leq K_n}e^{\theta
\sum_{i=1}^{|L_a(k)|}Y_{k,i}}\PP
\bigl( R_a^c\mid \mathbf{d}_a \bigr)
\biggr]e^{-\theta t}
\\
&\leq&\prod_{2\leq k\leq
\alpha_n} \biggl(1+\frac{\theta}{k(1-\upi_1^{(n)})-\theta}
\biggr)e^{-\theta
t}o \biggl( \frac{\log^{10}n}{n} \biggr),
\end{eqnarray*}
where the last inequality follows from inequality~(\ref{pr-R1}) in
Lemma~\ref{lem-R1}.
Thus taking $\theta= 1-\upi_1^{(n)}$ gives
\begin{eqnarray*}
\PP \bigl( T_a(\alpha_n)-\bT_a(2)\geq t,
\mathcal Q \bigr) &\leq&\prod_{2\leq k\leq
\alpha_n} \biggl(1+
\frac{1}{k-1} \biggr)e^{-(1-\upi_1^{(n)})
t}o \biggl( \frac{\log^{10}n}{n} \biggr)
\\
&\leq&\alpha_n e^{-(1-\upi_1^{(n)})t}o \biggl( \frac{\log^{10}n}{n} \biggr) =
e^{-(1-\upi_1^{(n)})t}o \biggl( \frac{\log^{13}n}{n} \biggr).
\end{eqnarray*}

 Since $d_a\geq\dmin$, we can easily deduce that
\[
\bigl( \bT_a(2)\mid\mathcal Q \bigr) \leq_{\mathrm{st}}
\operatorname{Exp} \bigl(\dmin\bigl(1-\upi_1^{(n)}\bigr)\bigr) +
\operatorname{Exp}\bigl(1-\upi_1^{(n)}\bigr),
\]
with these two
exponentials being independent and independent of $\mathcal Q$. Hence
we have
\begin{eqnarray*}
\PP \bigl( \bT_a(2)\geq t\mid\mathcal Q \bigr)&\leq& \int
_{t}^\infty \dmin\bigl(1-\upi_1^{(n)}
\bigr) \bigl(e^{-(1-\upi_1^{(n)})x}- e^{-\dmin(1-\upi
_1^{(n)})x} \bigr)
\\
&\leq& \dmin e^{-(1-\upi_1^{(n)})t}.
\end{eqnarray*}
Thus
\begin{eqnarray*}
\PP \bigl( T_a(\alpha_n)\geq t, \mathcal Q \bigr) &
\leq&e^{-(1-\upi_1^{(n)})t}o \biggl( \frac{\log^{13}n}{n} \biggr).
\end{eqnarray*}
Similar to the case where $R_a$ holds (by fixing a constant $\varepsilon
$ small enough and using that for $n$ sufficiently large $\upi_1^{(n)}
\leq q_1+\varepsilon$ for $n$ large enough), we get
\begin{eqnarray*}
\PP \bigl( T_a(\alpha_n)\geq x\log n+\ell, \mathcal Q
\bigr) &\leq&o \biggl( \frac{\log^{13}n}{n^{1+(1-q_1-\varepsilon)C}} \biggr)=o\bigl(n^{-1}\bigr).
\end{eqnarray*}

Putting all the above arguments together, and considering the three
disjoint cases $(R'_a)^c$ hold, $R_a$ holds and $R'_a\cap R_a^c $ holds
(in which case either $I_a <\alpha_n$ or $I_a \geq\alpha_n$), we
conclude that
\[
\PP \bigl( T_a(\alpha_n \wedge I_a)\geq
x \log n + \ell \bigr) \leq o\bigl(e^{-\dmin(1-q_1)\ell}\bigr)+o\bigl(n^{-1}
\bigr)+ 1-\PP\bigl(R'_a\bigr).
\]
To complete the proof it suffices to use Lemma~\ref{lem-R1}.
\end{pf}

We can now finish the proof of the upper bound in the case $\dmin\geq
2$. By Proposition~\ref{lem-up-rq}, and Lemma~\ref{lem-up-bis2} applied
to $\ell= \frac{\log n}{\dmin(1-q_1)}$, we obtain that
for a uniformly chosen vertex $a$ and any $\varepsilon>0$, we have
%
\begin{eqnarray}
\label{prop:upT22}&& \PP \biggl( \infty> T_a(\beta_n \wedge
I_a) \geq \biggl(\frac{1}{2(\nu
-1)} + \frac{1}{\dmin(1-q_1)} \biggr)
(1+\varepsilon)\log n \biggr)
\nonumber
\\[-8pt]
\\[-8pt]
\nonumber
&&\qquad= o\bigl(n^{-1}\bigr).
\end{eqnarray}
Indeed, the above probability can be bounded above by
\begin{eqnarray*}
&&\PP \biggl( T_a(\alpha_n \wedge I_a)
\geq \frac{1+\varepsilon}{\dmin(1-q_1)}\log n \biggr)
\\
&&\qquad{}+ \PP \biggl( T_a(\beta _n) - T_a(
\alpha_n) \geq\frac{1+\varepsilon}{2(\nu-1)}\log n \Big\mid I_a \geq
\alpha_n \biggr),
\end{eqnarray*}
and this is $o(n^{-1})$ by the above cited results.

Applying equation (\ref{prop:upT22}) (and Lemma~\ref{lem-up-bis2}) and
a union bound over $a$, we obtain
%
\begin{eqnarray}
\label{eq:flood-diam2} &&\PP \biggl(\forall a, T_a(\beta_n \wedge
I_a) \leq \biggl(\frac{1}{2(\nu-1)}+\frac{1}{\dmin(1-q_1)} \biggr) (1+
\varepsilon)\log n \biggr)
\nonumber
\\[-8pt]
\\[-8pt]
\nonumber
&&\qquad= 1-o(1).
\end{eqnarray}
Hence by Proposition~\ref{prop-up}, we have w.h.p. (for $\dmin\geq2$)
\[
\frac{\diam_w(G(n, (d_i)_1^n))}{\log n}\leq(1+\varepsilon) \biggl(\frac
{1}{\nu-1} +
\frac{1}{1-q_1}\ind[\dmin=2] + \frac{2}{\dmin}\ind[\dmin \geq3] \biggr).
\]

This proves the bound on the diameter. To obtain the upper bound for
the flooding time, we use equation (\ref{eq:flood-diam2}) and proceed
as above by applying Proposition~\ref{lem-up-rq} and Lemma~\ref
{lem-up-bis2} applied to $\ell= \varepsilon\log n$, to obtain that for a
uniformly chosen vertex~$b$, we have
%
\begin{equation}
\label{eq:flood-diam2bis} \PP \biggl(T_b(\beta_n \wedge
I_b) \leq \biggl(\frac{1+\varepsilon}{2(\nu
-1)}+\varepsilon\biggr)\log n
\biggr) = 1-o(1).
\end{equation}
Clearly, equations~(\ref{eq:flood-diam2}) and (\ref{eq:flood-diam2bis})
imply that w.h.p.
\begin{eqnarray*}
&&\frac{\flood_w(a,G(n, (d_i)_1^n))}{\log n}\\
&&\qquad\leq(1+\varepsilon) \biggl(\frac
{1}{\nu-1} +
\frac{1}{2(1-q_1)}\ind[\dmin=2] + \frac{1}{\dmin}\ind [\dmin\geq3] \biggr).
\end{eqnarray*}
Similarly, (\ref{eq:flood-diam2bis}) and Proposition~\ref{lem-up-rq}
give an upper bound for (\ref{th:dist}).

The proof of the upper bound in this case is now complete.

\subsection{Proof of the upper bound in the case $d_{\min}=1$}\label{sec:upp1}

In this section, we will need some results on the $2$-core of the
graph. Basic definitions and needed results are given in Appendix \ref
{sec:fpp-2core}.

We denote by $\cC_a$ the event that $a$ is connected to the $2$-core
of $G_n \sim G(n, (d_i)_1^n)$. It is well known (cf. Appendix~\ref{sec:fpp-2core}) that the condition $\nu>1$ ensures that the 2-core of
$G_n$ has size $\Omega(n)$, w.h.p.
We consider the graph $\tG_n(a)$ obtained from $G_n$ by removing all
vertices of degree one except $a$ until no such vertices exist.
If the event $\cC_a$ holds, $\tG_n(a)$ consists of the 2-core of $G_n$
and the unique path (empty if $a$ belongs to the $2$-core) from $a$ to
the $2$-core. While, if the event $\cC_a^c$ holds, then the graph $\tG
_n(a)$ is the union of the $2$-core of $G_n$ and the isolated vertex $a$.

In order to bound the weighted distance between two vertices $a$ and
$b$, in what follows, we will consider two cases depending on whether
both the vertices $a$ and $b$ are connected to the 2-core (i.e., the
events $\cC_a$ and $\cC_b$ both hold), or both vertices $a$ and $b$
belong to the same tree component of the graph. In the former case, we
will show how to adapt the analysis we made in the case $\dmin\geq2$
to this case. And in the latter case, we directly bound the diameter of
all the tree components of the graph.

First note that $\tG_n(a)$ can be constructed by means of the
configuration model with a new degree sequence $\td$ (cf. Appendix~\ref{sec:fpp-2core}) with $\td_i\geq2$ for all $i\neq a$. Consider the
exploration process on the graph $\tG_n(a)$, and denote by $\tT_a(i)$
the first time the ball $\tB_w(a,t)$ in $\tG_n(a)$ reaches size $i+1$.
Also, $\tilde I_a$ is defined similar to $I_a$ for the graph $\tG
_n(a)$. We need the following lemma.

\begin{lemma}\label{lem-up-bis3}
For a uniformly chosen vertex $a$, any $x>0$ and any $\ell=O(\log n)$,
we have
\[
\PP \bigl( \tT_a(\alpha_n \wedge\tilde
I_a)\geq x \log n + \ell \bigr) \leq o\bigl(n^{-1}\bigr)
+ o\bigl(e^{-(1-\lambda_*)\ell}\bigr).
\]
\end{lemma}

\begin{pf}
First note that if $\cC_a$ does not hold, that is, if $a$ is not
connected to the 2-core, we will have $\tilde{I}_a = 0$ (since
$\td_a=0$), and there is nothing to prove. Now the proof follows the
same lines as in the proof of Lemma~\ref{lem-up-bis2}. Note that
conditional on $\cC_a$, we have $\td_a\geq1$, hence by Lemma~\ref{lem-R1}, we have $\PP(R_a\mid\cC_a, \tilde{\bold{d}})\geq1-o(\log
^{10} n/n)$, and similarly for $R'_a$.
The only difference we have to highlight here, compared to the proof of
Lemma~\ref{lem-up-bis2}, is that conditional on $R_a\cap\cC_a$, we have
$\tilde{S}_a(j)\geq1+k$ for all $j\in\tilde{L}_a(k)$, where $\tilde
{S}_a(j)$ and $ \tilde{L}_a(k)$ are defined in the same way as $S_a(j)$
and $ L_a(k)$ for the graph $\tG(a)$. Take now $\theta=1-\underline
{\tilde{\pi}}_1^{(n)}$ in the Chernoff bound, used in the proof of
Lemma~\ref{lem-up-bis2}, where $\tilde{\upi}^{(n)}$ is defined as $\upi
^{(n)}$ for the degree sequence $(\td^{(n)}_1, \ldots, \td^{(n)}_{\tilde
{n}})$. The rest of the proof of Lemma~\ref{lem-up-bis2} can then be
easily adapted to obtain the same result, provided we replace
$2(1-q_1)$ by $(1-\lambda_*)$, which is precisely the statement of the
current lemma. (Note that $\lambda_* = \tilde q_1$; cf. Appendix~\ref{sec:fpp-2core}.)
\end{pf}

By Proposition~\ref{lem-up-rq} applied to the graph $\tG_n(a)$ (note
that $\tilde{\nu} = \nu$; cf. see Appendix~\ref{sec:fpp-2core}) and
Lemma~\ref{lem-up-bis3} applied to $\ell= \frac{\log n}{1-\lambda_*}$,
we obtain that
for a uniformly chosen vertex $a$ and any $\varepsilon>0$, we have
%
\begin{equation}\qquad
\label{prop:upT3} \PP \biggl( \infty> \tT_a(\beta_n \wedge
\tilde I_a) \geq \biggl(\frac
{1}{2(\nu-1)} + \frac{1}{1-\lambda_*}
\biggr) (1+\varepsilon)\log n \biggr) = o\bigl(n^{-1}\bigr).
\end{equation}
Indeed the above probability can be bounded above by
\begin{eqnarray*}
&&\PP \biggl( \tT_a(\alpha_n \wedge\tilde
I_a) \geq \frac{1+\varepsilon}{1-\lambda_*}\log n \biggr)
\\
&&\qquad{}+ \PP \biggl( \tT _a(\beta_n) - \tT_a(
\alpha_n) \geq\frac{1+\varepsilon}{2(\nu-1)}\log n \Big\mid\tilde I_a
\geq\alpha_n \biggr),
\end{eqnarray*}
and this is $o(n^{-1})$ by the above cited results.

Applying equation~(\ref{prop:upT3}) (and Lemma~\ref{lem-up-bis3}) and a
union bound over $a$, we obtain
%
\begin{equation}\qquad
\label{eq:flood-diam3} \PP \biggl(\forall a, \tT_a(\beta_n
\wedge\tilde I_a) \leq \biggl(\frac{1}{2(\nu-1)}+\frac{1}{1-\lambda_*}
\biggr) (1+\varepsilon)\log n \biggr) = 1-o(1).
\end{equation}

To obtain the upper bound for the flooding time and the typical
distance, we use equation (\ref{eq:flood-diam3}) and proceed as above
by using Lemma~\ref{lem-up-bis3} applied to $\ell= \varepsilon\log n$,
to obtain that for a uniformly chosen vertex $b$, we have
%
\begin{equation}
\label{eq:flood-diam3bis} \PP \biggl(\tT_b(\beta_n \wedge \tilde
I_b) \leq \biggl(\frac{1+\varepsilon}{2(\nu-1)}+\varepsilon \biggr)\log n
\biggr) = 1-o(1).
\end{equation}

Clearly, equation (\ref{eq:flood-diam3}) together with Proposition~\ref
{prop-up} [since $\tT_a(k)\geq T_a(k)$ for all $k$], imply the desired
upper bound on the giant component of $G_n$ and also on every
component containing a cycle, that is, connected to $2$-core.

At this point, we are only left to bound the (weighted) diameter of the
tree components.
In particular, the following lemma completes the proof.

\begin{lemma}
For two uniformly chosen vertices $a, b$ and any $\varepsilon>0$, we have
\[
\PP \biggl( \frac{1+\varepsilon}{1-\lambda_*} \log n < \dist_w(a, b) < \infty, \cC_a^c, \cC_b^c \biggr) = o
\bigl(n^{-2}\bigr).
\]
\end{lemma}

\begin{pf}
We consider the graph $\tG_n(a,b)$ obtained from $G_n$ by removing vertices
of degree less than two, except $a$ and $b$, until no such vertices exist.
As shown in Appendix~\ref{sec:fpp-2core}, the random graph $\tG_n(a,b)$
can be still obtained by a configuration model, and has the same
asymptotic parameters as the random graph $\tG_n(a)$ in the proof of
the previous lemma. We denote again by $\td$, the degree sequence of
the random graph $\tG_n(a,b)$. Also, $\tilde{T}_a$ and $\tilde I_a$ are
defined similarly for the graph $\tG_n(a,b)$.

Trivially, we can assume $\td_a = 1$ and $\td_b = 1$. Otherwise, either
they are not in the same component, and so $\dist_w(a,b) =\infty$, or
one of them is in the $2$-core; that is, one of the two events $\cC_a$
or $\cC_b$ holds.
Consider now the exploration process started at $a$ until time $k^*$
which is
the first time either a node with forward-degree (strictly) larger than
one appears or
the time that the unique half-edge adjacent to $b$ is chosen by the
process. Let $v^*$ be the node chosen at $k^*$.
{ Note that $\td_{v^*} = 1$ if and only if the half-edge incident to
$b$ is chosen at $k^*$.}
We have
\begin{eqnarray*}
&&\PP \biggl( \frac{1+\varepsilon}{1-\lambda_*} \log n < \dist_w(a, b)
< \infty,
\cC_a^c, \cC_b^c \biggr)
\\
&&\qquad= \PP \biggl( \tT_a\bigl(k^*\bigr) > \frac{1+\varepsilon}{1-\lambda_*} \log n,
v^* = b, \td_a = \td_b =1 \biggr)
\\
&&\qquad\leq\PP \biggl( \tT_a\bigl(k^*\bigr) > \frac{1+\varepsilon}{1-\lambda_*} \log n,
v^* = b \Big\mid \td_a = \td_b =1 \biggr)
\\
&&\qquad= \PP \biggl( \tT_a\bigl(k^*\bigr) > \frac{1+\varepsilon}{1-\lambda_*} \log n
\Big\mid \td_a = \td_b =1 \biggr)
\\
&&\qquad\quad{}\times\PP ( \td_{v^*} = 1 \mid\td_{v^*} \neq2,
\td_a = \td _b =1 ) = o\bigl(n^{-2}\bigr).
\end{eqnarray*}
To prove\vspace*{1pt} the last equality above, first note $\PP(\td_{v^*} = 1 \mid\td
_{v^*} \neq2, \td_a = \td_b =1 ) =O(\frac{1}n)$, this holds since $\nu
= \tilde\nu>1$ and $v^*$ will be chosen before $o(n)$ steps, that is,
$k^* =o(n)$ [we will indeed prove something much stronger, that $k^* =
O(\log n)$; cf. Lemma~\ref{lem-bound-clusters} in the next section].
Second, note that $ \PP ( \tT_a(k^*) > \frac{1+\varepsilon}{1-\lambda
_*} \log n \mid \td_a = \td_b =1  )=o(1/n)$. This follows by the
same argument as in the proof of Lemma~\ref{lem-up-bis3} applied to $\tG
_{n}(a,b)$, and by setting $\ell= \frac{(1+\varepsilon)\log n}{1-\lambda_*}$.
\end{pf}

The proof of the upper bound in this case is now complete by taking a
union bound over all $a$ and $b$.
We end this section by presenting the proof of Propositions~\ref
{prop-up} and~\ref{lem-up-rq} in the next subsection.

\subsection{Proof of Propositions~\texorpdfstring{\protect\ref{prop-up}}{3.1} and~\texorpdfstring{\protect\ref{lem-up-rq}}{3.2}}
We start this section by giving some preliminary results that we will
need in the proof of Propositions~\ref{prop-up} and~\ref{lem-up-rq}.

\begin{lemma}\label{lem-upp}
Let $\uD^{(n)}_i$ be i.i.d. with distribution $\upi^{(n)}$. For any
$\eta< \nu$,
there is a constant $\gamma> 0$ such that for $n$ large enough, we have
%
\begin{equation}
\label{eq-upp} \PP \bigl(\uD^{(n)}_1+\cdots+
\uD^{(n)}_k\leq k \eta \bigr)\leq e^{-\gamma k}.
\end{equation}
\end{lemma}

\begin{pf}
Let $D^*$ be a random variable with distribution $\PP(D^* = k) =
q_k$ given in equation~(\ref{eq:defq}) so that $\EE[D^*]=\nu$.
Let $\phi(\theta) = \EE[e^{-\theta D^*}]$. For any
$\varepsilon>0$, there exists $\theta_0>0$ such that for any $\theta\in
(0,\theta_0)$, we have
$\log\phi(\theta)< (- \nu+ \varepsilon)\theta$.
By Condition \ref{cond-dil} and the fact that $\beta_n \Delta_n =
o(n)$, that is, $\sum_{i=n-\beta_n+1}^n d_{(i)}^{(n)} = o(n)$, we have for
any $\theta>0$,
$\lim_{n \rightarrow\infty} \uphi^{(n)}(\theta) = \phi(\theta) $,
where $\uphi^{(n)}(\theta)=\break \EE[e^{-\theta\uD^{(n)}_1}]$. Also, for
$\theta>0$,
\[
\PP \bigl(\uD^{(n)}_1+\cdots+\uD^{(n)}_k
\leq\eta k \bigr) \leq \exp \bigl(k \bigl(\theta\eta+\log \uphi^{(n)}(
\theta) \bigr) \bigr).
\]
Fix $\theta<\theta_0$, and let $n$ be sufficiently large so that
$\log\uphi^{(n)}(\theta)\leq\log\phi(\theta)+\varepsilon$. This yields
\begin{eqnarray*}
\PP \bigl(\uD^{(n)}_1+\cdots+\uD^{(n)}_k
\leq\eta k \bigr) &\leq& \exp \bigl(k \bigl(\theta\eta+\log \phi(\theta) +
\varepsilon\theta \bigr) \bigr)
\\
&\leq& \exp \bigl(k\theta (\eta-\nu+2 \varepsilon ) \bigr),
\end{eqnarray*}
which completes the proof.\vadjust{\goodbreak}
\end{pf}

The following lemma is the main step in the proof of both the propositions.

\begin{lemma}\label{lem-R}
For any $\varepsilon>0$, define the event
\[
R''_a:= \biggl\{ S_a(k)
\geq\frac{\nu-1}{1+\varepsilon}k,\mbox{ for all } \alpha_n\leq k\leq
\beta_n \biggr\}.
\]
For a uniformly chosen vertex $a$, we have $\PP ( R''_a \mid I_a
\geq\alpha_n  )\geq
1-o(n^{-5})$.
\end{lemma}

 Before giving the proof of this lemma, we recall the
following basic result and one immediate corollary (for the proof see,
e.g.,~\cite{KM2010}, Theorem~1):

\begin{lemma}\label{lem-bin}
Let $n_1, n_2 \in\NN$ and $p_1, p_2 \in(0,1)$. We have $\Bin(n_1,p_1)
\leq_{\mathrm{st}} \Bin(n_2,p_2)$ if and only if the following conditions hold:
\begin{longlist}[(ii)]
\item[(i)] $n_1 \leq n_2$;
\item[(ii)] $(1-p_1)^{n_1} \geq(1-p_2)^{n_2}$.
\end{longlist}
\end{lemma}

In particular, we have:

\begin{corollary} \label{cor-bin}
If $x \leq y = o(n)$, we have (for $n$ large enough)
\[
x - \Bin(x,\sqrt{x/n}) \leq_{\mathrm{st}} y - \Bin(y, \sqrt{y/n}).
\]
\end{corollary}

\begin{pf}
By the above lemma, it is sufficient to show
$(x/n)^{x/2} \geq(y/n)^{y/2} $, and this is true because $s^s$ is
decreasing near zero (for $s < e^{-1}$).
\end{pf}

 Now we go back to the proof of Lemma~\ref{lem-R}.

\begin{pf*}{Proof of Lemma~\ref{lem-R}}
By Lemmas \ref{lem-coupl} and \ref{lem-upp}, for any $\varepsilon> 0$,
$k \geq\alpha_n$ and $n$ large enough, we have
\[
\PP \biggl( \hd_a(1) + \cdots + \hd_a(k) \leq
\frac{\nu}{1+\varepsilon/2}k \biggr) \leq e^{-\gamma k} = o\bigl(n^{-6}\bigr).
\]
We infer that with probability at least $1 - o(n^{-6})$, for any $k\leq
\beta_n$,
\[
\frac{\nu-1}{1+\varepsilon/2}k < d_a + \hd_a(1) + \cdots +
\hd_a(k) - k < (k+1) \Delta_n = o(n).
\]

 By the union bound over $k$, we have with
probability at least $1 - o(n^{-5})$ that for all $\alpha_n \leq k \leq
\beta_n$,
%
\begin{equation}
\label{ineq:S} \frac{\nu-1}{1+\varepsilon/2}k < \hS_a(k) < (k+1)
\Delta_n = o(n).
\end{equation}
Hence in the remainder of the proof we can assume that the above
condition is satisfied.

By Lemma~\ref{lem:upX}, Corollary~\ref{cor-bin} and inequalities~(\ref{ineq:S}),
conditioning on $\hS_a(k)$ and $\{I_a \geq k\}$, we have
\begin{eqnarray*}
\bigl( S_a(k) \mid\{I_a \geq k\} \bigr) &
\geq_{\mathrm{st}}& \frac{\nu
-1}{1+\varepsilon/2} k - \Bin \biggl(\frac{\nu-1}{1+\varepsilon/2} k,
\sqrt{ \biggl(\frac{\nu-1}{1+\varepsilon/2} k \biggr) \Big/ n} \biggr)
\\
&\geq_{\mathrm{st}}& \frac{\nu-1}{1+\varepsilon/2} k - \Bin (\nu k, \sqrt{\nu k / n} ).
\end{eqnarray*}

By Chernoff's inequality, since $k\sqrt{k/n} = o(k/\sqrt{\alpha_n})$,
we have
\[
\PP \bigl( \Bin(\nu k, \sqrt{\nu k /n})\geq k/\sqrt{\alpha_n} \bigr)
\leq\exp \bigl(- \tfrac{1}{3} k/\sqrt{\alpha_n} \bigr) = o
\bigl(n^{-6}\bigr).
\]
Moreover, conditioned on $\{I_a \geq k\}$, we have with probability at
least $1-o(n^{-6})$,
\[
S_a(k) \geq\frac{\nu-1}{1+\varepsilon/2} k - \frac{k}{\sqrt{\alpha_n}} \geq
\frac{\nu-1}{1+\varepsilon} k,
\]
for $n$ large enough. Defining
\[
R''_a(k):= \biggl\{
S_a(k) \geq\frac{\nu-1}{1+\varepsilon} k \biggr\} \qquad\mbox{for }
\alpha_n\leq k \leq\beta_n,
\]
so that $R''_a = \bigcap_{k=\alpha_n}^{\beta_n} R''_a(k)$, we have
%
\begin{equation}
\PP \bigl( R''_a(k) \mid
I_a \geq k \bigr) \geq1- o\bigl(n^{-6}\bigr).
\end{equation}

Thus, by using the fact that $R''_{a}(k-1) \subset \{ I_a \geq k
 \}$, we get
\begin{eqnarray*}
\PP \bigl( R''_a \mid
I_a \geq\alpha_n \bigr) &=& 1 - \PP \Biggl(\bigcup
_{k=\alpha_n}^{\beta_n} R''_a(k)^c
\Big\mid I_a \geq\alpha_n \Biggr)
\\
&=& 1 - \PP \Biggl(R''_a(
\alpha_n)^c \cup \bigcup_{k=\alpha_n+1}^{\beta_n}
\bigl(R''_a(k)^c \cap
R''_{a}(k-1) \bigr) \Big\mid
I_a \geq\alpha_n \Biggr)
\\
&\geq& 1 - \PP \Biggl(R''_a(
\alpha_n)^c \cup \bigcup_{k=\alpha_n+1}^{\beta
_n}
\bigl(R''_a(k)^c \cap \{
I_a \geq k \} \bigr) \Big\mid I_a \geq
\alpha_n \Biggr)
\\
&\geq&1 - \sum_{k=\alpha_n}^{\beta_n} \PP \bigl(
R''_a(k)^c \mid
I_a \geq k \bigr) \geq1-o\bigl(n^{-5}\bigr),
\end{eqnarray*}
which completes the proof.
\end{pf*}

We are now in position to provide the proof of both the propositions.

\begin{pf*}{Proof of Proposition~\ref{prop-up}}
Fix two vertices $u$ and $v$. We can assume that $T_u(\beta_n),T_v(\beta
_n)<\infty$,
that is, $I_u,I_v\geq\beta_n$. Otherwise the statement of the proposition
holds trivially for $u$ and $v$. Note that $\dist_w(u,v) \leq T_u(\beta
_n) + T_{v}(\beta_n)$ is equivalent to
\[
B_w\bigl(u,T_u(\beta_n)\bigr) \cap
B_w\bigl(v,T_v(\beta_n)\bigr) \neq
\varnothing.
\]
Hence, to prove the proposition we need to bound the probability that
$B_w(v,T_v(\beta_n))$ does not intersect $B_w(u,T_u(\beta_n))$.

First consider the exploration process for
$B_w(u,t)$ until reaching $t=T_u(\beta_n)$. We know by
Lemma~\ref{lem-R} that with probability at least $1-o(n^{-5})$,
\[
S_u(\beta_n) \geq\bigl(\nu-1-o(1)\bigr)
\beta_n.
\]
[In other words, there are at least $(\nu-1-o(1)) \beta_n$ half-edges in
$B_w(u,\break T_u(\beta_n))$.]

Next, begin exposing $B_w(v,t)$. Each matching adds a uniform half-edge to
the neighborhood of $v$. Therefore, the probability that
$B_w(v,T_v(\beta_n))$ does not intersect $B_w(u,T_u(\beta_n))$ is at most
\[
\biggl(1 - \frac{(\nu-1-o(1))\beta_n}{m^{(n)}} \biggr)^{\beta_n} \leq \exp\bigl[-
\bigl(9-o(1)\bigr)\log n\bigr] < n^{-4}
\]
for large $n$ (recall that $\beta_n^2 = \frac{9\lambda n\log n}{\nu
-1}$). The union bound over $u$ and $v$ completes the proof.
\end{pf*}

\begin{pf*}{Proof of Proposition~\ref{lem-up-rq}}
Conditioning on the event $R''_a$ defined in Lemma~\ref{lem-R}, we have
for any $\alpha_n\leq k\leq\beta_n$,
\[
T_a(k+1)-T_a(k)\leq_{\mathrm{st}} Y_k
\sim\operatorname{Exp}\bigl(S_a(k)\bigr)\leq_{\mathrm{st}}\operatorname {Exp}
\biggl(\frac{\nu
-1}{1+\varepsilon} k \biggr),
\]
and all the $Y_k$'s are independent.

Letting $s = \sqrt{\alpha_n}$, for $n$ large enough we obtain that
\begin{eqnarray*}
\EE \bigl[ e^{s (T_a(\beta_n) - T_a(\alpha_n))} \mid R''_a
\bigr] &\leq& \prod_{k=\alpha_n}^{\beta_n-1} \biggl(1 +
\frac{s}{({(\nu
-1)k}/{(1+\varepsilon)}) - s} \biggr) \\
&\leq&\prod_{k=\alpha_n}^{\beta_n-1}
\biggl(1 + \frac{s(1+2\varepsilon)}{(\nu
-1)k} \biggr)
\\
&\leq& \exp \Biggl[\frac{s(1+2\varepsilon)}{\nu-1} \sum_{k=\alpha_n}^{\beta
_n-1}
\frac{1}{k} \Biggr]\\
& \leq&\exp \biggl[ \frac{s(1+ 3\varepsilon) \log n}{2(\nu-1)} \biggr].
\end{eqnarray*}
By Markov's inequality,
\begin{eqnarray*}
&&\PP \biggl( T_a(\beta_n) - T_a(
\alpha_n) \geq\frac{(1+4\varepsilon)\log
n}{2(\nu-1)} \Big\mid I_a \geq
\alpha_n \biggr)
\\
&&\qquad\leq1 - \PP\bigl(R''_a\bigr) + \EE
\bigl[e^{s(T_a(\beta_n) - T_a(\alpha_n))} \mid R''_a \bigr]
\exp \biggl(- \frac{s(1+ 4\varepsilon) \log n}{2(\nu
-1)} \biggr)
\\
&&\qquad\leq\exp \biggl(- \frac{s\varepsilon\log n}{2(\nu- 1)} \biggr)+o\bigl(n^{-5}\bigr)= o
\bigl(n^{-1}\bigr),
\end{eqnarray*}
which concludes the proof.
\end{pf*}

\section{Proof of the lower bound}\label{sec:fpp-lower}

In this section we present the proof of the lower bound for Theorem~\ref
{thm-fpp-main}.
To prove the lower bound, it suffices to show that for any $\varepsilon>
0$, there
exists w.h.p. two vertices $u$ and $v$ such that
\[
\dist_w(u,v) >
(1-\varepsilon)
\biggl(\frac{1}{\nu-1} + \frac{2}{\Gamma(\dmin)} \biggr) \log n.
\]

As in the proof of the upper bound, the proof will be different
depending whether $\dmin=1$ or $\geq2$. So we start this section by
proving some preliminary results, including some new notation and
definitions, that we will need in the proof for these cases, and then
divide the end of the proof into two cases.

Fix a vertex $a$ in $G_n \sim G(n, (d_i)_1^n)$, and consider the
exploration process, defined in Section~\ref{sec:explor}.
Recall that $\bT_a(1)$ is the first time when the ball centered at $a$
contains a vertex of forward-degree at least two (i.e., degree at least
$3$); cf. equation~(\ref{eq:defbT}).
To simplify the notation, we denote by $C_a$ the ball centered at $a$
containing exactly one node (possibly in addition to $a$) of degree at
least $3$:
%
\begin{equation}
\label{eq:Ca}C_a:= B_w\bigl(a,
\bT_a(1)\bigr).
\end{equation}

Note that there is a vertex $u$ (of degree $d_u \geq3$) in $C_a$ which
is not in any ball $B_w(a,t)$ for $t< \bT_a(1)$ and we have $\max_{v\in
C_a}\dist_w(a,v)=\dist_w(a,u)$. We define the degree of $C_a$ as
%
\begin{equation}
\deg(C_a) = d_a+d_u-2.
\end{equation}
Remark that at time $\bT_a(1)$ of the exploration process defined in
Section~\ref{sec:explor} starting from $a$, we have at most $\deg(C_a)$
free half-edges, that is, the list $L$ contains at most $\deg(C_a)$
half-edges. [We have the equality if the tree excess until time $\bT
_a(1)$ is zero.]
The following lemma shows that the size of $C_a$ is relatively small.

\begin{lemma}\label{lem-bound-clusters}
Consider a random graph $G(n, (d_i)_1^n)$ where the degrees $d_i$
satisfy Condition \ref{cond-dil}. There exists a constant $M>0$,
independent of $n$, such that w.h.p. for all the nodes $a$ of the
graph, we have $|C_a| \leq M\log n$.
\end{lemma}

\begin{pf}
We consider the exploration process, defined in Section~\ref{sec:explor}, starting from a uniformly chosen vertex $a$, and use the
coupling of the forward-degrees we described in Section~\ref{sse:fpp-coupling}. Recall in particular that each forward-degree $\hd
(i)$ conditioned on the previous forward-degrees is stochastically
larger than a random variable with distribution $\upi^{(n)}$. This
shows that, at each step of the exploration process, the probability of
choosing a node of degree at most two (forward-degree one or zero) will
be at most $\upi_0^{(n)}+\upi_1^{(n)} < 1-\varepsilon$, for some $\varepsilon
>0$ (note that the asymptotic mean of $\upi^{(n)}$ is $\nu$, and by
assumption $\nu>1$). We conclude that there exists a constant $M>0$
such that for all large $n$,
$\PP(|C_a| > M \log n) = o(n^{-1})$.
The union bound over $a$ completes the proof.
\end{pf}

For two subsets of vertices $U,W \subset V$, the (weighted) distance
between $U$ and $W$ is defined as usual,
\[
\dist_w(U,W):= \min \bigl\{ \dist_w(u,w) \mid u
\in U, w \in W \bigr\}.
\]

For two nodes $a,b$, define the event $\HH_{a,b}$ as
%
\begin{equation}
\HH_{a,b}:= \biggl\{\frac{1-\varepsilon}{\nu-1} \log n < \dist_w
(C_a,C_b) < \infty \biggr\}.
\end{equation}

Note that $\frac{\log n}{\nu-1}$ is the typical distance, so the left
inequality in the definition of the above event means that $C_a$ and
$C_b$ have the right typical distance in the graph [modulo a factor
$(1-\varepsilon)$]. The right inequality simply means that $a$ and $b$
belong to the same connected component. The following proposition is
the crucial step in the proof of the lower bound, the proof of which is
postponed to the end of this section.

\begin{proposition}\label{prop-low-dist}
Consider a random graph $G(n, (d_i)_1^n)$ with i.i.d. rate one
exponential weights on its edges. Suppose that the degree sequence
$(d_i)_1^n$ satisfies Condition \ref{cond-dil}. Assume that the number
of nodes with degree one satisfy $u_1^{(n)} = o(n)$, and let $a$ and
$b$ be two distinct vertices such that $\deg(C_a) = O(1)$, and $\deg
(C_b)=O(1)$. Then for all $\varepsilon> 0$,
\[
\PP (\HH_{a,b} ) = 1- o(1).
\]
Furthermore, the same result holds without the condition $\deg(C_a) =
O(1)$ [resp., $\deg(C_b) = O(1)$] if the node $a$ (resp., $b$) is chosen
uniformly at random.
\end{proposition}

Note that in particular, Proposition~\ref{prop-low-dist} is still valid
when $a$ and $b$ are chosen uniformly
at random and hence provides a lower bound for (\ref{th:dist}).

Assuming the above proposition, we now show that:
\begin{longlist}[(ii)]
%
\item[(i)] If the minimum degree $\dmin\geq2$, then there are pairs
of nodes $a$ and $b$ of degree $\dmin$ such that $\HH_{a,b}$ holds, and
in addition, the closest nodes to each with forward-degree at least two
is at distance at least $(1-\varepsilon)\log n/(\dmin(1-q_1))$ w.h.p., for
all $\varepsilon>0$.
\item[(ii)] If the minimum degree $\dmin=1$, then there are pairs of
nodes of degree one such that $\HH_{a,b}$ holds, and in addition, the
closest node to each which belongs to the $2$-core is at least
$(1-\varepsilon)\log n/(1-\lambda_*)$ away w.h.p., for all $\varepsilon>0$.
\end{longlist}

This will finish the proof of the claimed lower bound.

\subsection{Proof of the lower bound in the case \texorpdfstring{$d_{\min}\geq2$}{$d_{\min}>=2$}}
Let $V^*$ be the set of all vertices of degree $\dmin$. We call a
vertex $u$ in $V^*$ \textit{good} if $\bT_u(1) $ is at least $\frac
{1-\varepsilon} {\dmin(1-q_1)} \log n$, that is,
\[
\bT_u(1) \geq\frac{1-\varepsilon} {\dmin(1-q_1)} \log n,
\]
and if in addition, $ \deg(C_u) \leq K$ for a constant $K$ chosen as
follows. Let $\hD$ be a random variable with the size-biased
distribution, that is, $\PP(\hD= k) = q_k$. The constant $K$ is chosen
in order to have with positive probability $\hD\leq K - \dmin+1$
conditioned on the event that $\hD\geq2$, that is,
%
\begin{equation}
y=y_K:= \PP (\hD\leq\dmin- 1 + K \mid\hD\geq2 ) > 0.
\end{equation}
It is easy to verify that such $K$ exists since $\nu>1$.

It will be convenient to consider the two events in the definition of
good vertices separately, namely, for a vertex $u\in V^*$, define
%
\begin{eqnarray}
\cE_u &:=& \biggl\{ \bT_u(1) \geq\frac{1-\varepsilon} {\dmin(1-q_1)}
\log n \biggr\} \quad\mbox{and}
\\
\cE'_u &:=& \bigl\{ \deg(C_u) \leq K
\bigr\}.
\end{eqnarray}
We note that in the case $\dmin\geq3$, the event $\cE_u$ for $u \in
V^*$ is equivalent to having a weight greater than $\frac{1-\varepsilon}{
\dmin} \log n$ on all the $\dmin$ edges connected to $u$, and clearly,
the two above events $\cE'_u$ and $\cE_u$ are independent
(conditionally on $u\in V^*$, i.e., $d_u=\dmin$).

For $u \in V^*$, let $A_u$ be the event that $u$ is good, $A_u:= \cE_u
\cap\cE'_u$, and let $Y$ be the total number of good vertices, $Y:=\sum_u \ind_{A_u}$. In the following, we first obtain a bound for the
expected value of $Y$, and then use the second moment inequality to
show that w.h.p. $Y =\Omega(n^{\varepsilon})$.

Consider the exploration process defined in Section~\ref{sec:explor},
starting from a node $u \in V^*$.
At the beginning, each step of the exploration process is an
exponential with parameter $\dmin$ (since there are $\dmin$
yet-unmatched half-edges adjacent to the explored vertices). In each
step, the probability that the new half-edge of the list $L$ does not
match to the other half-edge of $L$ 
is at least $1-1/n$. This follows by observing that there are at least
$n$ yet-unmatched half-edges (by $\nu> 1$), and by using Lemma~\ref
{lem-bound-clusters} (which says that before $M\log n$ steps the
exploration process meets a vertex of forward-degree at least two). By
the forward-degree coupling arguments of Section~\ref{sse:fpp-coupling}, the probability that a new matched node be of
forward-degree one is at least $\bpi_1^{(n)}$. This shows that, with
probability at least $(1-1/n)\bpi_1^{(n)}$ the exploration process adds
a new node of forward-degree one.
This shows that the first step in the exploration process a vertex of
forward-degree at least two is added will be stochastically bounded
below by a geometric random variable of parameter $(1-1/n)\bpi
_1^{(n)}$. Each step takes rate $\dmin$ exponential time. Therefore,
\begin{eqnarray*}
\PP(\cE_u) &=& \PP \biggl(\bT_u(1) \geq
\frac{1-\varepsilon} {\dmin
(1-q_1)} \log n \biggr)
\\
&\geq& \PP \biggl(\operatorname{Exp} \bigl(\dmin \bigl(1-(1-1/n)\bpi_1^{(n)}
\bigr) \bigr) \geq\frac{1-\varepsilon} {\dmin(1-q_1)} \log n \biggr).
\end{eqnarray*}
In the last inequality we used the fact that a sum of a geometric (with
parameter $\pi$) number of
independent exponential random variables of rate $\mu$ is
distributed as an exponential random variable of rate $(1-\pi)\mu$.
Note that this in particular shows that
\[
\PP(\cE_u) \geq\bigl(1-o(1)\bigr) \exp\bigl(-(1-\varepsilon) \log
n\bigr) = \bigl(1-o(1)\bigr) n^{-1 +
\varepsilon}.
\]

By using the coupling arguments of Section~\ref{sse:fpp-coupling} (and
by using Lemma~\ref{lem-bound-clusters}), to bound the forward-degrees
from above (and below) by i.i.d. random variables having distributions
$\bpi^{(n)}$ (and $\upi^{(n)}$) and then using the fact that the
asymptotic distributions of both $\bpi^{(n)}$ and $\upi^{(n)}$
coincides with the size biased distribution $\{q_k\}$, we have
$\PP(\cE'_u \mid\cE_u) = (1 \pm o(1))y$.
We conclude that $\PP(A_u) = \PP(\cE'_u \mid\cE_u) \PP(\cE_u) \geq (1
\pm o(1))y n^{-1 + \varepsilon}$.

This shows that
\[
\EE[Y] = \sum_{u\in V^*} \PP(A_u) \geq
\bigl(1 \pm o(1)\bigr) y p_{\dmin} n^{\varepsilon}.
\]

 Note that above, we used Condition~\ref{cond-dil} which
implies that $|V^*| = (1\pm o(1))p_{\dmin}n$.

We now show that $\Var(Y) = o(\EE[Y]^2)$. Applying the Chebyshev
inequality, this will show that $Y \geq\frac{2}{3} yp_{\dmin}
n^{\varepsilon}$ with high probability.

For any pair of vertices $u,v \in V^*$ such that $C_u \cap C_v =
\varnothing$, conditioning on $A_u$ does not have much effect on the
asymptotic of the degree distribution (by Lemma~\ref
{lem-bound-clusters} the size of each component $C_u$ is at most $M\log
n$), and hence, we deduce by the coupling argument of Section~\ref{sse:fpp-coupling} that for $u$ and $v$ such that $C_u \cap C_v =
\varnothing$,
\[
\PP(A_v \cap A_u) = \bigl(1 \pm o(1)\bigr)
\PP(A_u) \PP(A_v).
\]

We infer
\begin{eqnarray*}
\Var(Y) &=& \EE\bigl[Y^2\bigr] - \EE[Y]^2 = \EE\biggl[
\sum_{u,v \in V^*} \ind_{A_u} \ind _{A_v}
\biggr] - \EE[Y]^2
\\
&=& \EE \biggl[\sum_{u,v \in V^*\dvtx C_u \cap C_v \neq\varnothing} \ind_{A_u} \ind
_{A_v} + \sum_{u,v \in V^*\dvtx C_u \cap C_v = \varnothing} \ind_{A_u}
\ind _{A_v} \biggr] - \EE[Y]^2
\\
&=& \EE\biggl[\sum_{u \in V^*} \ind_{A_u} \sum
_{v \in V^*\dvtx C_u \cap C_v \neq
\varnothing} \ind_{A_v} + \sum
_{u,v \in V^*\dvtx C_u \cap C_v = \varnothing} \ind_{A_u} \ind_{A_v} \biggr] -
\EE[Y]^2
\\
&\leq&(K+1) (M \log n) \EE[Y] +\EE\biggl[ \sum_{u,v \in V^*\dvtx C_u \cap C_v =
\varnothing}
\ind_{A_u} \ind_{A_v} \biggr] - \EE[Y]^2
\\
&=& o\bigl(\EE[Y]^2\bigr).
\end{eqnarray*}
In the inequality above, we used Lemma~\ref{lem-bound-clusters} to
bound w.h.p. the size of all $C_w$ by $M\log n$ (for some large enough
$M$) for any node $w$ in the graph, and used the fact that if the event
$A_u$ holds, then there are at most $K$ edges out-going from $C_u$.
Each of the vertices $v$ with the property that $C_u\cap C_v\neq
\varnothing$ should be either already on $C_u$ or connected with a path
consisting only of vertices of degree two to $C_u$ (in which case, this
path should belong to $C_v$). A simple analysis then shows that the
number of nodes $v$ with the property that $C_u \cap C_v \neq\varnothing
$ is bounded by $(K+1) M \log n$, and the inequality follows.

This completes the proof of the fact that $Y \geq\frac{2}{3} p_{\dmin}
yn^{\varepsilon}$ with high probability.

We consider first the flooding time, and obtain the corresponding lower
bound. Let $Y'$ denote the number of good vertices that are at distance
at most $\frac{1-\varepsilon} {\dmin(1-q_1)} \log n + \frac{1-\varepsilon
}{\nu-1} \log n$ from a vertex $a$ (chosen uniformly at random). It is
clear that the lower bound follows by showing that $Y' < Y$ with high
probability, that is, $Y-Y' >0$ w.h.p. To show this, we will bound the
expected value of $Y'$ and use Markov's inequality.

Since by Condition~\ref{cond-dil}, $V^*$ has size linear in $n$ by
applying Proposition~\ref{prop-low-dist}, we obtain that for a
uniformly chosen vertex $u\in V^*$, conditioning on $A_u$, we have $\PP
(\HH_{a,u}) = 1-o(1)$. Indeed, the two events $\HH_{a,u}$ and $\cE_u$
are independent, and conditioning on $\cE'_u$ is the same as
conditioning on $\deg(C_u) \leq K = O(1)$. Therefore, for a uniformly
chosen vertex $u$ in $V^*$, we have
\[
\PP \bigl(A_u \cap\HH^c_{a,u} \bigr) = o
\bigl(\PP(A_u)\bigr),
\]
where $\HH^c_{a,u}$ denotes the complementary event of $\HH_{a,u}$,
that is, the event that $\HH_{a,u}$ does not occur.
Thus, a straightforward calculation shows that
$\EE[Y'] = o(\EE[Y]) = o(n^{\varepsilon})$.
By Markov's inequality, we conclude that $Y' \leq\frac{1}{3} p_{\dmin
}y n^{\varepsilon}$ w.h.p., and hence $Y-Y'$ is w.h.p. positive. This
implies the existence of a vertex $u$ whose distance from $a$ is at
least $ ( \frac{1}{\nu-1} + \frac{1}{\dmin(1-q_1)} )(1-\varepsilon
) \log n $. Hence for any $\varepsilon> 0$ we have w.h.p.
\begin{eqnarray*}
&&\flood_w(a,G_n) \\
&&\qquad\geq\max_{ u \in V^*}
\dist_w(a,u)
\\
&& \qquad\geq\biggl( \frac{1}{\nu-1} + \frac{1}{2(1-q_1)}\ind[\dmin=2] +
\frac{1}{\dmin}\ind [\dmin\geq3] \biggr) (1-\varepsilon) \log n.
\end{eqnarray*}

We now turn to the proof of the lower bound for the (weighted)
diameter of the graph. The proof will follow the same strategy as
for the flooding time, but this time we need to consider the pairs of
good vertices.

Let $R$ denote the number of pairs of distinct good vertices. Recall
we proved above that w.h.p. $Y \geq\frac{2}{3} \EE[Y]$. Thus
\[
R = Y(Y-1) \geq\frac{2\EE[Y]}{3} \biggl( \frac{2\EE[Y]}{3}-1\biggr) >
\frac{1}{4} \EE[Y]^2.
\]

The probabilities that $u$ and $v$ are both good and $\HH_{u,v}$ does
not happen can be bounded as follows:
%
\begin{eqnarray}
\label{eq:intersect}
\nonumber
\PP \bigl(A_u\cap A_v \cap
\HH^c_{u,v} \bigr) &=& \PP(A_u \cap
A_v) \PP\bigl(\HH^c_{u,v} \mid
A_u, A_v\bigr)
\\
&=& \PP(A_u \cap A_v) \PP \bigl(
\HH^c_{u,v} \mid\deg(C_u)\leq K,
\deg(C_v) \leq K \bigr)
\nonumber
\\[-8pt]
\\[-8pt]
\nonumber
&& \mbox{(We used the independence of $\HH_{u,v}$ and $\cE
_u$ and $\cE_v$)}
\\
&=& o\bigl(\PP(A_u \cap A_v)\bigr).\nonumber
\end{eqnarray}
The last equality follows from Proposition~\ref{prop-low-dist}, since
$C_u$ and $C_v$ are of degree $O(1)$.

To conclude, consider $R'$ the number of pairs of good vertices that
are at distance at most $(1-\varepsilon)(2\frac{\log n} {\dmin(1-q_1)} +
\frac{\log n}{\nu-1})$. By using equation~(\ref{eq:intersect}), we have
$\EE R' = o(\EE[Y]^2)$. Applying Markov's inequality, we obtain that
w.h.p. $R' \leq\frac{1}{6} (\EE[Y])^2 $, and thus $R-R'$ is w.h.p.
positive. This implies that for any $\varepsilon> 0$, we have w.h.p.
\begin{eqnarray*}
\diam_w(G_n) &\geq& \max_{ u,v \in V^* }
\dist_w(u,v)
\\
&\geq& \biggl( \frac{1}{\nu-1} + \frac{1}{1-q_1}\ind[\dmin=2] +
\frac
{2}{\dmin} \ind[\dmin\geq3] \biggr) (1-\varepsilon) \log n.
\end{eqnarray*}

\subsection{Proof of the lower bound in the case $d_{\min}=1$}\label{sec:low1c}
Consider the $2$-core algorithm, and stop the process the first time
the number of nodes of degree one drops below $n^{1-\varepsilon/2}$. Let
$V^*$ be the set of all nodes of degree one at this time. We denote by
$\tG_n(V^*)$ the graph constructed by configuration model on the set of
remaining nodes (this is indeed the $V^*$-augmented 2-core). Observe
that proving the lower bound on the graph $\tG_n(V^*)$ gives us the
lower bound on $G_n$.

Since $|V^*| = o(n/\log n)$, and the $2$-core has linear size in $n$,
w.h.p. the degree sequence of $\tG_n(V^*)$ has the same asymptotic as
the degree sequence in the $2$-core of $G_n$; see Appendix~\ref{sec:fpp-2core}, Lemma~\ref{lem-core-augmented} for more details. In
particular, we showed in Appendix~\ref{sec:fpp-2core} that for the
size-biased degree sequence of the 2-core's degree distribution, we
have $\tilde q_1 =\lambda_*$, and for its mean, we have $\tilde\nu =
\nu$.

Repeating the coupling arguments of Section~\ref{sse:fpp-coupling} and
defining $\tilde{\bpi}^{(n)}$ (similar to the definition of $\bpi
^{(n)}$) for the degree sequence of $\tG_n(V^*)$, we infer that $\tilde
{\bpi}_1^{(n)} \to\lambda_*$.

As before, call a vertex $u$ in $V^*$ \textit{good} if both the events
$\cE_u$ and $\cE'_u$ hold. Recall the definition of the two events
%
\begin{eqnarray}
\cE_u &:=& \biggl\{ \tilde{\bT}_u(1) \geq
\frac{1-\varepsilon} {1-\lambda
_*} \log n \biggr\} \qquad\mbox{and}
\\
\cE'_u &:=& \bigl\{ \deg(C_u) \leq K
\bigr\}.
\end{eqnarray}
Here the constant $K\geq2$ is chosen with the property that $\tq_K >
0$ ($\tq$ is the size-biased probability mass function corresponding to
the $2$-core; cf. Appendix~\ref{sec:fpp-2core}), and $\tilde{\bT}_u$ is
defined similar to $\bT_u$ for the graph $\tG_n(V^*)$.

Consider the exploration process starting from a node $u \in V^*$. At
the beginning, each step of the exploration process is an exponential
of rate one, and the probability that each new matched node be of
forward-degree exactly one is at least $\tilde{\bpi}_1^{(n)}$.
Similar to the case of $\dmin=2$, we obtain
\begin{eqnarray*}
\PP(A_u) &\geq& \bigl(1 \pm o(1)\bigr) \tq_K \PP
\bigl(\operatorname{Exp} \bigl(1-\tilde {\bpi}_1^{(n)} \bigr)
\geq{1-\varepsilon} {(1-\lambda_*)} \log n \bigr)
\\
&=& \bigl(1 \pm o(1)\bigr) \tq_K \exp\biggl(-(1-\varepsilon)
\frac{1-\lambda_*}{1-\tilde
{\bpi}_1^{(n)}}\log n\biggr)
\\
&=&\bigl(1 \pm o(1)\bigr) \tq_K n^{-1+\varepsilon}.
\end{eqnarray*}
This shows that
\[
\EE[Y] = \sum_{u\in V^*} \PP(A_u)\geq
n^{1-\varepsilon/2}\bigl(1 \pm o(1)\bigr) \tq_K n^{-1+\varepsilon} =
\bigl(1 \pm o(1)\bigr)\tq_K n^{\varepsilon/2}.
\]
Similarly, we obtain that $\Var(Y) = o(\EE[Y]^2)$, and the rest of the
proof follows similar to the precedent case by using Proposition~\ref
{prop-low-dist} for $\tilde G_n(V^*)$. Note that in $\tilde G_n(V^*)$,
the number of vertices of degree one is $o(n) = o(|\tilde G_n(V^*)|)$,
and thus Proposition~\ref{prop-low-dist} can be applied.

At the present we are only left to prove Proposition~\ref{prop-low-dist}.

\subsection{Proof of Proposition~\texorpdfstring{\protect\ref{prop-low-dist}}{4.2}}
In this section we present the proof of Proposition~\ref{prop-low-dist}.
It is shown in \cite{janson09b,MolReed98} that the giant component of a
random graph $G(n, (d_i)^n_1)$ for $(d_i)_1^n$ satisfying Condition~\ref
{cond-dil} contains w.h.p. all but $o(n)$ vertices [since $\nu> 1$ and
$u_0^{(n)} + u_1^{(n)} = o(n)$]. This immediately shows that $\PP(\dist
_w (C_a,C_b) < \infty) = 1-o(1)$. Define $t_n:= \frac{1-\varepsilon}{2(\nu
-1)} \log n $. So to prove the proposition, we need to prove that $\dist
_w(C_a,C_b)$ is lower bounded by $t_n$ w.h.p. in the case where either
$\deg(C_a)=O(1)$ [resp., $\deg(C_b)=O(1)$] or $a$ (resp., $b$) is chosen
uniformly at random.

In the case where $a$ is chosen uniformly at random, it is easy to
deduce, by using Markov's inequality, that we have w.h.p. $\deg(C_a)
\leq\log n$. Indeed, this is true since $\deg(C_a)$ is asymptotically
distributed as $ ( D + \hD- 1 \mid\hD\geq2 )$, where $\hD$
is a random variable with the size-biased distribution,
and $D$ is independent of $\hD$ with the degree distribution $\{p_k\}$.
[To show this, one can use the coupling argument of
Section~\ref{sse:fpp-coupling} to bound $\deg(C_a)$ stochastically from above.]
And, since this latter random variable has finite moment (by
Condition~\ref{cond-dil}), by applying Markov's inequality, we obtain
w.h.p. $\deg(C_a) \leq\log n$. This shows that in all cases stated in
the proposition, we can assume that $\deg(C_a) \leq\log n$ and $\deg
(C_b) \leq\log n$.

We now consider the exploration process defined in Section~\ref{sec:explor}
starting from~$C_a$; that is, we start the exploration
process with $B=C_a$, and apply the steps one and two of the process.
In a similar way we defined $T_a(i)$, we define $T_{C_a}(i)$ to be the
time of the $i$th step in this continuous-time exploration process.
Similarly, let $\hd_{C_a}(i)$ be the forward-degree of the vertex added
at $i$th exploration step for all $i\geq1$, and define
%
\begin{equation}
\hS_{C_a}(i):= \deg(C_a) + \hd_{C_a}(1) +
\cdots + \hd_{C_a}(i) - i,
\end{equation}
and define $S_{C_a}(i)$ similarly, so that we have $S_{C_a}(i) \leq\hS
_{C_a}(i)$. Note that $T_{C_a}(i)$ obviously satisfies
\[
T_{C_a}(i+1)-T_{C_a}(i) = \operatorname{Exp} \bigl(
S_{C_a}(i) \bigr) \geq _{\mathrm{st}} Y_i \sim
\operatorname{Exp} \bigl( \hS_{C_a}(i) \bigr),
\]
where the random variables $Y_i$ are all independent.

Also, we infer (by Lemma~\ref{lem-coupl}) that
%
\begin{equation}
\hS_{C_a}(i) \leq_{\mathrm{st}} \log n + \sum
_{j=1}^{i} \bD^{(n)}_j -i,
\end{equation}
where $\bD^{(n)}_j$ are i.i.d. with distribution $\bpi^{(n)}$.

Let $\bnu^{(n)}$ be the expected value of $\bD^{(n)}_1$ which is
\[
\bnu^{(n)}:= \sum_k k
\bpi_k^{(n)},
\]
and define
$z_n = \sqrt{n / \log n } $. We will show later that the two growing
balls in the exploration processes started from $C_a$ and $C_b$, for
$a$ and $b$ as in the proposition, will not intersect w.h.p. provided
that they are of size less than $z_n$. We now prove that $T_{C_a}(z_n)
\geq t_n$ with high probability.

 For this, let us define
\[
T'(k) \sim\sum_{i=1}^k
\operatorname{Exp} \Biggl( \log n+ \sum_{j=1}^{i}
\bD ^{(n)}_j - i \Biggr),
\]
where all the exponential variables in the above sum are independent,
such that by the above arguments, we have
\[
T_{C_a}(z_n) \geq_{\mathrm{st}} T'(z_n).
\]

We need the following lemma. [We define $\operatorname{Exp}(s):= +\infty$
for $s \leq0$.]

\begin{lemma}\label{lem-coupl-exp}
Let $X_1,\ldots,X_t$ be a random process adapted to a filtration $\mathcal
{F}_0 = \sigma[\o], \mathcal{F}_1,\ldots,\mathcal{F}_t$, and let $\mu_i =
\EE X_i$, $\bolds{\Sigma}_i = X_1+ \cdots+X_i$, $\Lambda_i = \mu_1 + \cdots
+ \mu_i$. Let $Y_i \sim\operatorname{Exp}(\Sigma_i)$, and $Z_i \sim\operatorname
{Exp}(\Lambda_i)$, where all exponential variables are independent.
Then we have
\[
Y_1 + \cdots + Y_t \geq_{\mathrm{st}}
Z_1 + \cdots + Z_t.
\]
\end{lemma}

\begin{pf}
By Jensen's inequality, it is easy to see that for positive random
variable $X$, we have
\[
\operatorname{Exp}(X) \geq_{\mathrm{st}} \operatorname{Exp}(\EE X).
\]
Then by induction, it suffices to prove that for a pair of random
variables $X_1$, $X_2$ we have $Y_1 + Y_2 \geq_{\mathrm{st}} Z_1 + Z_2 $. We have
\begin{eqnarray*}
\PP(Y_1 + Y_2 > s) &=& \EE_{X_1} \bigl[
\PP(Y_1 + Y_2 > s | X_1 )\bigr]
\\
&\geq& \EE_{X_1} \bigl[\PP\bigl(\operatorname{Exp}(X_1) +
\operatorname{Exp}(X_1 + \mu_2) > s\bigr)\bigr]
\\
&\geq& \PP(Z_1 + Z_2 >s).
\end{eqnarray*}
\upqed\end{pf}

We infer by Lemma~\ref{lem-coupl-exp},
\[
T'(z_n) \geq_{\mathrm{st}} \sum
_{i=0}^{z_n} \operatorname{Exp} \bigl(\log n + \bigl(\bnu
^{(n)} - 1\bigr) i \bigr) =: T^*(z_n),
\]
where all exponential variables are independent.

We now let $b_n:= \log n - (\bnu^{(n)} - 1)$, so that we have
\begin{eqnarray*}
&&\PP\bigl(T^*(z_n) \leq t\bigr)\\
&&\qquad \leq\int_{\sum x_i \leq t}e^{- \sum_{i=1}^{z_n}
((\bnu^{(n)} - 1)i + b_n)x_i}
\,dx_1 \cdots dx_{z_n} \prod_{i=1}^{z_n}
\bigl(\bigl(\bnu ^{(n)} - 1\bigr)i + b_n\bigr)
\\
&&\qquad= \int_{0 \leq y_1 \leq\cdots\leq y_{z_n} \leq t}e^ {-(\bnu^{(n)} - 1)
\sum_{i=1}^{z_n} y_i} e^{-b_n y_{z_n}}
\,dy_1 \cdots dy_{z_n} \prod_{i=1}^{z_n}
\bigl(\bigl(\bnu^{(n)} - 1\bigr)i + b_n\bigr),
\end{eqnarray*}
where $y_k = \sum_{i=0}^{k-1} x_{z_n-i}$.
Letting $y$ play the role of $y_{z_n}$, and accounting for all
permutations over $y_1, \ldots, y_{z_{n}-1}$ (giving each such variable
the range $[0, y]$), we obtain{
\begin{eqnarray*}
&&\PP\bigl(T^*(z_n) \leq t\bigr)\\
&&\qquad \leq\bigl(\bnu^{(n)} - 1
\bigr)^{z_n} \frac{\prod_{i=1}^{z_n} (i + {b_n}/{(\bnu^{(n)} - 1)})}{(z_n-1)!}
\\
&&\qquad\quad{}\times\int_0^t e^{-(\bnu^{(n)} - 1 + b_n) y} \biggl( \int
_{[0,y]^{z_n-1}} e^{-(\bnu^{(n)} - 1)\sum_{i=1}^{z_n-1}y_i} \,dy_1\cdots dy_{z_n-1}
\biggr) \,dy
\\
&&\qquad= z_n \frac{\prod_{i=1}^{z_n} (i + {b_n}/{\bnu^{(n)} -
1})}{z_n!}\bigl(\bnu^{(n)} - 1\bigr)
\\
&&\qquad\quad{}\times\int_0^t e^{-(\bnu^{(n)} - 1 + b_n) y} \Biggl( \prod
_{i=1}^{z_n-1} \int_0^y
\bigl(\bnu^{(n)} - 1\bigr) e^{-(\bnu^{(n)} - 1)y_i} \,dy_i \Biggr) \,dy
\\
&&\qquad= z_n \prod_{i=1}^{z_n}
\biggl(1 + \frac{b_n}{(\bnu^{(n)} - 1)i}\biggr) \bigl(\bnu^{(n)} - 1\bigr)
\\
&&\qquad\quad{}\times\int_0^t e^{-(\bnu^{(n)} - 1 + b_n) y} \bigl(1 -
e^{-(\bnu^{(n)}
- 1)y} \bigr)^{z_n-1} \,dy
\\
&&\qquad\leq c z_n^{{b_n}/{(\bnu^{(n)} - 1)}+1} \bigl(\bnu^{(n)} - 1\bigr) \int
_0^t e^{-(\bnu^{(n)} - 1 + b_n) y} \bigl(1 -
e^{-(\bnu^{(n)} - 1)y} \bigr)^{z_n-1} \,dy, 
\end{eqnarray*}
where $c>0$ is an absolute constant. Recall that $t_n=\frac{1-\varepsilon
}{2(\nu-1)} \log n$, and $z_n=\sqrt{n/\log n}$. Now we use the fact
that $ (1 - e^{-(\bnu^{(n)} - 1)y} )^{z_n-1} \leq e^{-n^{\alpha
}}$, for some $\alpha> 0$ and for all $0 \leq y \leq t_n$. We infer
\[
\PP\bigl(T^*(z_n) \leq t_n\bigr) \leq c \bigl(
\bnu^{(n)} - 1\bigr) z_n^{{b_n}/{(\bnu^{(n)} - 1)}+1} \int
_0^{t_n} e^{-n^\alpha} \,dy = o
\bigl(n^{-4}\bigr),
\]
since $b_n=O(\log n)$.
}
Hence, we have w.h.p.
\[
\bigl|B_w(C_a,t_n)\bigr| \leq z_n.
\]
[Here naturally, for $W\subseteq V$, we let $B_w(W,t) = \{b, \mbox
{such that} \dist_w(W,b)\leq t \}$.]

Similarly for $b$, and exposing $B_w(C_b,t_n)$, again w.h.p. we obtain a
set of size at most $z_n$. Now note that, because each matching is
uniform among the remaining half-edges, the probability of hitting
$B_w(C_a,t_n)$ is at most $\hS_{C_a}(z_n)/n$.

Let $\varepsilon_n:= \log\log n$. By Markov's inequality we have
\begin{eqnarray*}
\PP \bigl(\hS_{C_a}(z_n) \geq z_n
\varepsilon_n \bigr) \leq \EE\hS _{C_a}(z_n)/z_n
\varepsilon_n
= \frac{K + (\bnu^{(n)} - 1) (z_n+\lambda_n)}{z_n \varepsilon_n} = o(1).
\end{eqnarray*}

We conclude
\begin{eqnarray*}
&&\PP \bigl( B_w(C_a,t_n) \cap
B_w(C_b,t_n) \neq\varnothing \bigr)
\\
&&\qquad\leq \PP \bigl( \bigl|B_w(C_a,t_n)\bigr| >
z_n \bigr) + \PP \bigl( \bigl|B_w(C_b,t_n)\bigr|
> z_n \bigr)
\\
&&\qquad\quad{}+ \PP \bigl(\hS_{C_a}(z_n) \geq z_n
\varepsilon_n \bigr) + \varepsilon _n z_n^2/n\\
&&\qquad= o(1).
\end{eqnarray*}
This completes the proof of Proposition~\ref{prop-low-dist}.

\begin{appendix}\label{app}
\section{Structure of the $2$-core}\label{sec:fpp-2core}

%
%

The $k$-core of a given graph $G$ is the largest induced subgraph of
$G$ with minimum vertex degree at least $k$. The $k$-core of an
arbitrary finite
graph can be found by removing vertices of degree less than $k$, in an
arbitrary order, until no such vertices exist.

Consider now a random graph $G_n \sim G^*(n, (d_i)_1^n)$ where the
degree sequence $(d_i)_1^n$ satisfies Condition \ref{cond-dil}. In the
process of constructing a random graph $G_n$ by matching the
half-edges, the $k$-core can be found by successively removing the
half-edge of a node of degree less than $k$ followed by removing a
uniformly random half-edge from the set of all the remaining half-edges
until no such vertices (of degree less than $k$) remain. What remains
at this time is the $k$-core. Since these half-edges are unexposed, the
$k$-core edge set is uniformly random conditional on the $k$-core
half-edge set.
Let $k \geq2$ be a fixed integer, and $\Core_k^{(n)}$ be the $k$-core
of the graph $G_n \sim G^*(n, (d_i)_1^n)$.
For integers $l \geq0$ and $0 \leq r \leq l$, let $\pi_{lr}$ denote
the binomial probabilities
\[
\pi_{lr}(p) = \PP \bigl( \Bin(l,p) = r \bigr) =\pmatrix {l \cr r}
p^r (1-p)^{l-r}.
\]


We further define the functions
\[
h(p):= \sum_{r=k}^{\infty} \sum
_{l=r}^{\infty} r p_l \pi_{lr}(p)\quad
\mbox{and} \quad h_1(p):= \sum_{r=k}^{\infty}
\sum_{l=r}^{\infty} p_l
\pi_{lr}(p).
\]

\begin{theorem}[(Janson and Luczak~\cite{janluc07})]\label{thm:cm-core}
Consider a random graph $G(n, (d_i)_1^n)$ where the degree sequence
$(d_i)_1^n$ satisfies Condition \ref{cond-dil}. Let $k \geq2$ be
fixed, and let \textit{$\Core^{(n)}_k$} be the $k$-core of $G(n,
(d_i)_1^n)$. Let $\hp$ be the largest $p \leq1$ such that $\mu p^2 = h(p)$.
Assume $\hp> 0$, and further suppose that $\hp$ is not a local maximum
point of the function $h(p) - \mu p^2$. Then
\[
v \bigl(\Core_k^{(n)} \bigr)/n \stackrel{p} {
\rightarrow} h_1(\hp) > 0,\qquad  v_j \bigl(
\Core_k^{(n)} \bigr)/n \stackrel{p} {\rightarrow} \sum
_{l=j}^{\infty} p_l
\pi_{lj}(\hp)
\]
for $j\geq k$, and $e (\Core_k^{(n)} )/n \stackrel
{p}{\rightarrow} \mu\hp^2/2$.
\end{theorem}

From now on, we consider the case $k=2$, and denote by $\tG$ the
$2$-core of a graph $G$. In particular applying Theorem~\ref
{thm:cm-core} to the case $k=2$, we have
$ h(\hp):=\mu\hp- \sum_l l p_l \hp(1-\hp)^{l-1} = \mu\hp(1-G_q(1-\hp))$.
Recall from Theorem~\ref{thm:cm-core} that we have to solve the
equation $\mu\hp^2 = h(\hp)$. Thus we obtain $1-\hp= G_q(1-\hp)$, and
so $\hp= 1-\lambda$.

By \cite{lel-diff}, Theorem~10, the graph $\tG_n$ obtained from $G_n$
has the same distribution as a random graph constructed by the
configuration model on $\tilde{n}$ nodes with a degree sequence $\td
^{(n)}_1, \ldots, \td^{(n)}_{\tilde{n}}$ satisfying the following
properties:
\begin{eqnarray*}
\tilde{n}/n \stackrel{p} {\to} h_1(1-\lambda) &=&1-G_p(
\lambda )-(1-\lambda)G'_p(\lambda)
\\
&=& 1-G_p(\lambda)-\mu\lambda(1-\lambda)>0
\end{eqnarray*}
and
\begin{eqnarray*}
\bigl|\bigl\{i, \td^{(n)}_i=j\bigr\}\bigr|/n &\stackrel{p} {\to}&
\sum_{\ell=j}^\infty p_\ell\pmatrix {\ell
\cr j}(1-\lambda)^j\lambda^{\ell-j},\qquad j\geq2,
\\
\sum_i \td^{(n)}_i/n &
\stackrel{p} {\to}& \mu(1-\lambda)^2.
\end{eqnarray*}
It follows that the sequence $\{\td^{(n)}_1, \ldots, \td^{(n)}_{\tilde
{n}}\}$ satisfies
also Condition \ref{cond-dil} for some probability distribution
$\tilde{p}_k$ with mean $\tilde{\mu}$ (which can be easily calculated
from the two above properties).

Let $\tilde{q}$ be the size-biased probability mass function
corresponding to $\tilde{p}$. We now show that $\tilde q$ and $q$ have
the same mean. Indeed, denoting by $\tilde\nu$ the mean of $\tq$, we
see that $\tilde\nu$ is given by
%
\begin{eqnarray}
\tilde{\nu} &:=& \sum_k k
\tq_k = \frac{1}{\tilde{\mu}} \sum_{k}
k(k-1) \tilde{p}_k
\nonumber\\
&=& \frac{\sum_{k \geq2} k (k-1) \sum_{\ell\geq k} p_\ell
{\ell\choose k}(1-\lambda)^k\lambda^{\ell-k}}{\mu(1-\lambda)^2}
\nonumber
\\[-8pt]
\\[-8pt]
\nonumber
&=& \frac{\sum_{\ell} p_\ell\sum_{k \leq\ell} k (k-1) {\ell
\choose k}(1-\lambda)^k\lambda^{\ell-k}}{\mu(1-\lambda)^2}
\\
&=& \frac{\sum_{\ell} p_\ell\ell(\ell-1)}{\mu} = \nu.\nonumber
\end{eqnarray}
To find the diameter in the case $\dmin=1$, we also need to show that
$\tilde{q}_1 =\lambda_*$:
%
\begin{eqnarray}
\label{eq-aug-q1}
\tilde{q}_1 &=& \frac{2\tilde{p}_2}{\tilde{\mu}}=
\frac{2\sum_{\ell\geq2} p_\ell{\ell\choose2} (1-\lambda)^2\lambda^{\ell-2}}{\mu
(1-\lambda)^2}
\nonumber
\\[-8pt]
\\[-8pt]
\nonumber
&=& \frac{1}{\mu} G''_p(\lambda)
= G'_q(\lambda)=\lambda_*.
\end{eqnarray}

We will also need the following relaxation of the notion of $2$-core.
Let $G=(V,E)$ be a graph. For a given subset $W \subseteq V$, define
the $W$-augmented $2$-core to be the maximal induced subgraph of $G$
such that every vertex in $V\setminus W$ has degree at least two; that
is, the vertices in $W$ are not required to verify the minimum degree
condition in the definition of the $2$-core. The $W$-augmented $2$-core
of a graph $G$ will be denoted by $\tG(W)$.

It is easy to see that the $W$-augmented $2$-core of a random graph
$G_n \sim G^*(n, (d_i)_1^n)$, denoted by $\tG_n(W)$, can be found in
the same way as the $2$-core, except that now the termination condition
is that every node outside of $W$ must have degree at least $2$, since
the half-edges adjacent to a vertex in $W$ are exempt from this
restriction. The conditional uniformity
property thus evidently holds in this case as well; that is, for any
subset $W \subset V$, the $W$-augmented $2$-core is uniformly random,
conditional on the $W$-augmented $2$-core half-edge set. We will need
the following basic result, the proof of which is easy and can be found,
for example, in~\cite{fern07}, Lemma A.7.

\begin{lemma}\label{lem-core-augmented}
Consider a random graph $G_n \sim G(n, (d_i)_1^n)$ where the degree
sequence $(d_i)_1^n$ satisfies Condition \ref{cond-dil}. { For any
subset $W\subset V(G_n)$, and any \mbox{$w\in W$}, there exists $C>0$
(sufficiently large) so that we have
\[
\PP \bigl(e\bigl(\tG_n(W)\bigr)-e\bigl(\tG_n\bigl(W
\setminus\{w\}\bigr)\bigr) \leq C \log n \bigr) = 1-o\bigl(n^{-1}\bigr).
\]
}
\end{lemma}

Note that the above lemma implies (by removing one vertex
from $W$ at a time) that if $|W| = o(n/\log n)$, then w.h.p. the two
graphs $\tG_n$ and $\tG_n(W)$ have the same degree distribution asymptotic.

\section{The random graphs $G(n,p)$ and $G(n,m)$}\label{sec:Sko}

We derive the results for $G(n,p)$ and $G(n,m)$ from our results for
$G(n,\break (d_i)_1^n)$ by conditioning on the degree sequence. Indeed, we can
be more general and consider a random graph $G_n$ with $n$ vertices
labeled $[1,n]$ and some random distribution of the edges such that any
two graphs on $[1,n]$ with the same degree sequence have the same
probability of being attained by $G_n$. Equivalently, conditioned on
the degree sequence, $G_n$ is a random graph with that degree sequence
of the type $G(n,(d_i)_1^n)$ introduced in the \hyperref[sec:fpp-intro]{Introduction}. We may
thus construct $G_n$ by first picking a random sequence $(d_i)_1^n$
with the right distribution, and then choosing a random graph
$G(n,(d_i)_1^n)$ for this $(d_i)_1^n$.

We assume that Condition \ref{cond-dil} holds in probability:

\begin{condition}\label{cond-dil-p}
For each $n$, let $\mathbf{d}^{(n)} = (d^{(n)}_i)_1^n$ be the random
sequence of vertex degrees of $G_n$ and $u_k^{(n)}$ be the random
number of vertices with degree $k$.
Then, for some probability distribution $(p_r)_{r=0}^{\infty}$ over
integers independent of $n$ and with finite mean $\mu:=\sum_{k\geq
0}kp_k\in(0,\infty)$, the following holds:
\begin{longlist}[(ii)]
\item[(i)] $u_k^{(n)}/n\stackrel{p}{\rightarrow} p_k$ for every $k\geq
1$ as $n \to\infty$;
\item[(ii)] For some $\varepsilon>0$, $\sum_{k=1}^\infty k^{2+\varepsilon
}u_k^{(n)}=O_p(n)$;
\end{longlist}
\end{condition}

We first show that for $G(n,p)$ and $G(n,m)$, with $np\to\mu\in
(0,\infty)$ and $2m/n\to\mu$, Condition \ref{cond-dil-p} holds with
$(p_k)$ a Poisson distribution with parameter $\mu$, that is, $p_k =
e^{-\mu}\frac{\mu^k}{k!}$.
Indeed the fact that Condition~\ref{cond-dil-p}(i) holds with such
$(p_k)$ follows by elementary estimates of mean and variance done in
Example~6.35 of \cite{jlr} or Theorem~3.1 in \cite{bollobas}. Showing that
Condition~\ref{cond-dil-p}(ii) holds can be done by similar
arguments. Consider $G(n,p)$ [a similar argument holds for $G(n,m)$],
we have for all $k\geq0$ and for $n$ sufficiently large,
\[
n^{-1} \EE u_k^{(n)} = \pmatrix{n-1\cr k}
p^k(1-p)^{n-1-k}< (\mu+1)^k/k!.
\]
Thus $n^{-1} \sum_{k=1}^\infty k^{2+\varepsilon} \EE u_k^{(n)} = O(1)$,
and Condition~\ref{cond-dil-p}(ii) holds.

The following lemma is similar to Lemma~8.2 in \cite{janson-2008}.

\begin{lemma}
If Condition \ref{cond-dil-p} holds, we may, by replacing the random
graph $G_n$ by other random graphs $G'_n$ with the same distribution,
assume that the random graphs are defined on a common probability space
and that Condition \ref{cond-dil} holds a.s.
\end{lemma}

\begin{pf}
If only Condition \ref{cond-dil}(i) was required, this lemma would be a
direct consequence of the Skorohod coupling theorem (Theorem~3.30, \cite
{kallenberg}) for the random sequence $(u^{(n)}_k)_{k=1}^\infty$ in the
space $\RR_+^\infty$. 
We now explain how to incorporate Conditions \ref{cond-dil}(ii).
Condition \ref{cond-dil-p} implies that it is possible to find an
increasing sequence $C_j$ for $j\geq1$ diverging to infinity so that
considering the sets
\[
A_j = \Biggl\{ (x_k)_{k=1}^\infty\in
\RR_+^\infty, \sum_{k=1}^\infty
x_k< \infty, \sum_{k=1}^\infty
k^{2+\varepsilon}x_k\leq C_j\sum
_{k=1}^\infty x_k 
 \Biggr\},
\]
we have for all $n$, $\PP ((u^{(n)}_k)\in A_j  )\geq
1-(2j)^{-1}$ (note that $\sum_{k=1}^\infty u^{(n)}_k=n$). 
Let $q^{(n)}_j=\PP ((u^{(n)}_k)\in A_j  )$ so that
$q^{(n)}_{j+1}\geq q^{(n)}_j\geq1-(2j)^{-1}$ for all $j\geq1$.
For each~$\ell$, we define an associated finite sequence
$j^{(n)}_i(\ell)$ for $i=1,\ldots, k^{(n)}(\ell)$ such that
$j^{(n)}_1(\ell)=1$ and for $i\geq1$,
$j^{(n)}_{i+1}(\ell)=\min\{ j\geq j^{(n)}_i(\ell), q^{(n)}_{j}-q^{(n)}_{j^{(n)}_i(\ell)}\geq\frac{1}{2\ell}\}$ if
$q^{(n)}_{j^{(n)}_i(\ell)}< 1-(2\ell)^{-1}$, and if
$q^{(n)}_{j^{(n)}_i(\ell)}\geq
1-(2\ell)^{-1}$, we set $k^{(n)}(\ell)=i$. Let
$\cJ^{(n)}(\ell)=\{j^{(n)}_1(\ell)=1,j^{(n)}_2(\ell),\ldots,j^{(n)}_{k^{(n)}(\ell)}(\ell)\}$. Note
that, since $q^{(n)}_\ell\geq1-(2\ell)^{-1}$, we have $k^{(n)}(\ell
)\leq\ell$.

We now explicitly construct a ``Skorohod coupling.'' Let $\theta$ be a
uniform random variable in $[0,1]$, and define the random variable
$J^{(n)}(\ell)$ by $J^{(n)}(\ell)=\min\{j\in\cJ^{(n)}(\ell), \theta
\leq q^{(n)}_{j}\}$ if $\theta\leq q^{(n)}_{k^{(n)}(\ell)}$, and if
$\theta>q^{(n)}_{k^{(n)}(\ell)}$, we set $J^{(n)}(\ell)=\infty$.
We set $j^{(n)}_0(\ell)=0$, $j^{(n)}_i(\ell)=\infty$ for $i>k^{(n)}(\ell
)$, $A_0=\varnothing$ and $A_\infty=\RR_+$.
With these definitions, we have for all $n$ and $i\geq1$,
$\PP(J^{(n)}(\ell)=j^{(n)}_i(\ell))=\PP ((u^{(n)}_k)\in
A_{j^{(n)}_i(\ell)}\setminus A_{j^{(n)}_{i-1}(\ell)}  )$.

For a given $\ell$ and for any $i\geq1$, we define the random variables
$\tilde{u}^{(n)}(i)=(\tilde{u}^{(n)}_k(i))_{k\in\NN}\in\RR_+^\infty$
having the law of $(u^{(n)}_k)$ conditioned on the event
$\{(u^{(n)}_k)\in A_{j^{(n)}_i(\ell)}\setminus
A_{j^{(n)}_{i-1}(\ell)} \}$. 
Note in particular that by construction, if $i\leq k^{(n)}(\ell)$, we
have $\PP ( (u^{(n)}_k)\in A_{j^{(n)}_i(\ell)}\setminus
A_{j^{(n)}_{i-1}(\ell)} )\geq(2\ell)^{-1}$. Hence if there exist
an infinite sequence of $n$ such that $i\leq k^{(n)}(\ell)$, then we
can apply the Skorohod coupling theorem and assume that, along this
subsequence, Condition \ref{cond-dil}(i) holds.

We can now combine this coupling with the following one: given
$\theta$ taken uniformly at random in $[0,1]$, take
$\ell=\lceil\frac{1}{2(1-\theta)}\rceil$, and consider
$\tilde{u}^{(n)}(J^{(n)}(\ell))$ which has the same law as the
original $u^{(n)}$. By construction, Condition \ref{cond-dil}(i)
holds. Moreover, we have by construction $J^{(n)}(\ell)\leq\ell$
since $q^{(n)}_{\ell}\geq1-(2\ell)^{-1}>\theta$, so that $\tilde
{u}^{(n)}(J^{(n)}(\ell))\in A_{\ell}$ and Condition \ref{cond-dil}(ii) holds.
\end{pf}
\end{appendix}

\section*{Acknowledgments}
This paper is a part of the first author's Ph.D. thesis \cite
{aminithesis11}. He is grateful to Remco van der Hofstad and Laurent
Massouli\'e who reviewed the thesis, and made many valuable comments
which helped improve the presentation of this paper. We also thank Omid
Amini for helpful comments and discussions.


\begin{thebibliography}{31}

\bibitem{aminithesis11}
\begin{bmisc}[auto:STB|2014/05/28|10:36:42]
\bauthor{\bsnm{Amini},~\bfnm{H.}\binits{H.}}
(\byear{2011}).
\bhowpublished{Epidemics and percolation in random networks.
Ph.D. thesis, ENS-INRIA.
Available at \url{http://www.di.ens.fr/\textasciitilde amini/Publication/Thesis.pdf}.}
\end{bmisc}
\bptok{imsref}%
\endbibitem

\bibitem{amdrle-fl}
\begin{barticle}[mr]
\bauthor{\bsnm{Amini},~\bfnm{Hamed}\binits{H.}},
\bauthor{\bsnm{Draief},~\bfnm{Moez}\binits{M.}} \AND
\bauthor{\bsnm{Lelarge},~\bfnm{Marc}\binits{M.}}
(\byear{2013}).
\btitle{Flooding in weighted sparse random graphs}.
\bjournal{SIAM J. Discrete Math.}
\bvolume{27}
\bpages{1--26}.
\bid{doi={10.1137/120865021}, issn={0895-4801}, mr={3032902}}
\end{barticle}
\bptok{imsref}%
\endbibitem

\bibitem{amlel-12}
\begin{barticle}[mr]
\bauthor{\bsnm{Amini},~\bfnm{Hamed}\binits{H.}} \AND
\bauthor{\bsnm{Lelarge},~\bfnm{Marc}\binits{M.}}
(\byear{2012}).
\btitle{Upper deviations for split times of branching processes}.
\bjournal{J. Appl. Probab.}
\bvolume{49}
\bpages{1134--1143}.
\bid{issn={0021-9002}, mr={3058993}}
\end{barticle}
\bptok{imsref}%
\endbibitem

\bibitem{amperes}
\begin{barticle}[mr]
\bauthor{\bsnm{Amini},~\bfnm{Hamed}\binits{H.}} \AND
\bauthor{\bsnm{Peres},~\bfnm{Yuval}\binits{Y.}}
(\byear{2014}).
\btitle{Shortest-{w}eight {p}aths in {r}andom {r}egular {g}raphs}.
\bjournal{SIAM J. Discrete Math.}
\bvolume{28}
\bpages{656--672}.
\bid{doi={10.1137/120899534}, issn={0895-4801}, mr={3192643}}
\end{barticle}
\bptok{imsref}%
\endbibitem

\bibitem{bencan78}
\begin{barticle}[mr]
\bauthor{\bsnm{Bender},~\bfnm{Edward~A.}\binits{E.~A.}} \AND
\bauthor{\bsnm{Canfield},~\bfnm{E.~Rodney}\binits{E.~R.}}
(\byear{1978}).
\btitle{The asymptotic number of labeled graphs with given degree sequences}.
\bjournal{J. Combin. Theory Ser. A}
\bvolume{24}
\bpages{296--307}.
\bid{mr={0505796}}
\end{barticle}
\bptok{imsref}%
\endbibitem

\bibitem{bhamidi08}
\begin{barticle}[mr]
\bauthor{\bsnm{Bhamidi},~\bfnm{Shankar}\binits{S.}}
(\byear{2008}).
\btitle{First passage percolation on locally treelike networks. {I}. {D}ense random graphs}.
\bjournal{J. Math. Phys.}
\bvolume{49}
\bpages{125218, 27}.
\bid{doi={10.1063/1.3039876}, issn={0022-2488}, mr={2484349}}
\end{barticle}
\bptok{imsref}%
\endbibitem

\bibitem{BHH10extreme}
\begin{barticle}[mr]
\bauthor{\bsnm{Bhamidi},~\bfnm{Shankar}\binits{S.}},
\bauthor{\bsnm{van~der Hofstad},~\bfnm{Remco}\binits{R.}} \AND
\bauthor{\bsnm{Hooghiemstra},~\bfnm{Gerard}\binits{G.}}
(\byear{2010}).
\btitle{Extreme value theory, {P}oisson--{D}irichlet distributions, and first passage percolation on random networks}.
\bjournal{Adv. in Appl. Probab.}
\bvolume{42}
\bpages{706--738}.
\bid{doi={10.1239/aap/1282924060}, issn={0001-8678}, mr={2779556}}
\end{barticle}
\bptok{imsref}%
\endbibitem

\bibitem{BHH09}
\begin{barticle}[mr]
\bauthor{\bsnm{Bhamidi},~\bfnm{Shankar}\binits{S.}},
\bauthor{\bsnm{van~der Hofstad},~\bfnm{Remco}\binits{R.}} \AND
\bauthor{\bsnm{Hooghiemstra},~\bfnm{Gerard}\binits{G.}}
(\byear{2010}).
\btitle{First passage percolation on random graphs with finite mean degrees}.
\bjournal{Ann. Appl. Probab.}
\bvolume{20}
\bpages{1907--1965}.
\bid{doi={10.1214/09-AAP666}, issn={1050-5164}, mr={2724425}}
\end{barticle}
\bptok{imsref}%
\endbibitem

\bibitem{BHH10}
\begin{barticle}[mr]
\bauthor{\bsnm{Bhamidi},~\bfnm{Shankar}\binits{S.}},
\bauthor{\bsnm{van~der Hofstad},~\bfnm{Remco}\binits{R.}} \AND
\bauthor{\bsnm{Hooghiemstra},~\bfnm{Gerard}\binits{G.}}
(\byear{2011}).
\btitle{First passage percolation on the {E}rd{\H o}s--{R}\'enyi random graph}.
\bjournal{Combin. Probab. Comput.}
\bvolume{20}
\bpages{683--707}.
\bid{doi={10.1017/S096354831100023X}, issn={0963-5483}, mr={2825584}}
\end{barticle}
\bptok{imsref}%
\endbibitem

\bibitem{bhamidi2012universality}
\begin{bmisc}[auto:STB|2014/05/28|10:36:42]
\bauthor{\bsnm{Bhamidi},~\bfnm{S.}\binits{S.}},
\bauthor{\bsnm{van~der Hofstad},~\bfnm{R.}\binits{R.}} \AND
\bauthor{\bsnm{Hooghiemstra},~\bfnm{G.}\binits{G.}}
(\byear{2012}).
\bhowpublished{Universality for first passage percolation on sparse random graphs. Preprint.
Available at \arxivurl{arXiv:1210.6839}.}
\end{bmisc}
\bptok{imsref}%
\endbibitem

\bibitem{bol80}
\begin{barticle}[mr]
\bauthor{\bsnm{Bollob{\'a}s},~\bfnm{B{\'e}la}\binits{B.}}
(\byear{1980}).
\btitle{A probabilistic proof of an asymptotic formula for the number of labelled regular graphs}.
\bjournal{European J. Combin.}
\bvolume{1}
\bpages{311--316}.
\bid{doi={10.1016/S0195-6698(80)80030-8}, issn={0195-6698}, mr={0595929}}
\end{barticle}
\bptok{imsref}%
\endbibitem

\bibitem{bollobas}
\begin{bbook}[mr]
\bauthor{\bsnm{Bollob{\'a}s},~\bfnm{B{\'e}la}\binits{B.}}
(\byear{2001}).
\btitle{Random Graphs},
\bedition{2nd} ed.
\bseries{Cambridge Studies in Advanced Mathematics}
\bvolume{73}.
\bpublisher{Cambridge Univ. Press},
\blocation{Cambridge}.
\bid{doi={10.1017/CBO9780511814068}, mr={1864966}}
\end{bbook}
\bptok{imsref}%
\endbibitem

\bibitem{Boldlv}
\begin{barticle}[mr]
\bauthor{\bsnm{Bollob{\'a}s},~\bfnm{B.}\binits{B.}} \AND
\bauthor{\bsnm{Fernandez~de~la Vega},~\bfnm{W.}\binits{W.}}
(\byear{1982}).
\btitle{The diameter of random regular graphs}.
\bjournal{Combinatorica}
\bvolume{2}
\bpages{125--134}.
\bid{doi={10.1007/BF02579310}, issn={0209-9683}, mr={0685038}}
\end{barticle}
\bptok{imsref}%
\endbibitem

\bibitem{boljanrio}
\begin{barticle}[mr]
\bauthor{\bsnm{Bollob{\'a}s},~\bfnm{B{\'e}la}\binits{B.}},
\bauthor{\bsnm{Janson},~\bfnm{Svante}\binits{S.}} \AND
\bauthor{\bsnm{Riordan},~\bfnm{Oliver}\binits{O.}}
(\byear{2007}).
\btitle{The phase transition in inhomogeneous random graphs}.
\bjournal{Random Structures Algorithms}
\bvolume{31}
\bpages{3--122}.
\bid{doi={10.1002/rsa.20168}, issn={1042-9832}, mr={2337396}}
\end{barticle}
\bptok{imsref}%
\endbibitem

\bibitem{CL03}
\begin{barticle}[mr]
\bauthor{\bsnm{Chung},~\bfnm{Fan}\binits{F.}} \AND
\bauthor{\bsnm{Lu},~\bfnm{Linyuan}\binits{L.}}
(\byear{2003}).
\btitle{The average distance in a random graph with given expected degrees}.
\bjournal{Internet Math.}
\bvolume{1}
\bpages{91--113}.
\bid{issn={1542-7951}, mr={2076728}}
\end{barticle}
\bptok{imsref}%
\endbibitem

\bibitem{ding09}
\begin{barticle}[mr]
\bauthor{\bsnm{Ding},~\bfnm{Jian}\binits{J.}},
\bauthor{\bsnm{Kim},~\bfnm{Jeong~Han}\binits{J.~H.}},
\bauthor{\bsnm{Lubetzky},~\bfnm{Eyal}\binits{E.}} \AND
\bauthor{\bsnm{Peres},~\bfnm{Yuval}\binits{Y.}}
(\byear{2010}).
\btitle{Diameters in supercritical random graphs via first passage percolation}.
\bjournal{Combin. Probab. Comput.}
\bvolume{19}
\bpages{729--751}.
\bid{doi={10.1017/S0963548310000301}, issn={0963-5483}, mr={2726077}}
\end{barticle}
\bptok{imsref}%
\endbibitem

\bibitem{fern07}
\begin{barticle}[mr]
\bauthor{\bsnm{Fernholz},~\bfnm{Daniel}\binits{D.}} \AND
\bauthor{\bsnm{Ramachandran},~\bfnm{Vijaya}\binits{V.}}
(\byear{2007}).
\btitle{The diameter of sparse random graphs}.
\bjournal{Random Structures Algorithms}
\bvolume{31}
\bpages{482--516}.
\bid{doi={10.1002/rsa.20197}, issn={1042-9832}, mr={2362640}}
\end{barticle}
\bptok{imsref}%
\endbibitem

\bibitem{grimkest84}
\begin{barticle}[mr]
\bauthor{\bsnm{Grimmett},~\bfnm{Geoffrey}\binits{G.}} \AND
\bauthor{\bsnm{Kesten},~\bfnm{Harry}\binits{H.}}
(\byear{1984}).
\btitle{First-passage percolation, network flows and electrical resistances}.
\bjournal{Z. Wahrsch. Verw. Gebiete}
\bvolume{66}
\bpages{335--366}.
\bid{doi={10.1007/BF00533701}, issn={0044-3719}, mr={0751574}}
\end{barticle}
\bptok{imsref}%
\endbibitem

\bibitem{hagpem98}
\begin{barticle}[mr]
\bauthor{\bsnm{H{\"a}ggstr{\"o}m},~\bfnm{Olle}\binits{O.}} \AND
\bauthor{\bsnm{Pemantle},~\bfnm{Robin}\binits{R.}}
(\byear{1998}).
\btitle{First passage percolation and a model for competing spatial growth}.
\bjournal{J. Appl. Probab.}
\bvolume{35}
\bpages{683--692}.
\bid{issn={0021-9002}, mr={1659548}}
\end{barticle}
\bptok{imsref}%
\endbibitem

\bibitem{janson99}
\begin{barticle}[mr]
\bauthor{\bsnm{Janson},~\bfnm{Svante}\binits{S.}}
(\byear{1999}).
\btitle{One, two and three times {$\log n/n$} for paths in a complete graph with random weights}.
\bjournal{Combin. Probab. Comput.}
\bvolume{8}
\bpages{347--361}.
\bid{doi={10.1017/S0963548399003892}, issn={0963-5483}, mr={1723648}}
\end{barticle}
\bptok{imsref}%
\endbibitem

\bibitem{janluc07}
\begin{barticle}[mr]
\bauthor{\bsnm{Janson},~\bfnm{Svante}\binits{S.}} \AND
\bauthor{\bsnm{Luczak},~\bfnm{Malwina~J.}\binits{M.~J.}}
(\byear{2007}).
\btitle{A simple solution to the {$k$}-core problem}.
\bjournal{Random Structures Algorithms}
\bvolume{30}
\bpages{50--62}.
\bid{doi={10.1002/rsa.20147}, issn={1042-9832}, mr={2283221}}
\end{barticle}
\bptok{imsref}%
\endbibitem

\bibitem{janson-2008}
\begin{barticle}[mr]
\bauthor{\bsnm{Janson},~\bfnm{Svante}\binits{S.}} \AND
\bauthor{\bsnm{Luczak},~\bfnm{Malwina~J.}\binits{M.~J.}}
(\byear{2008}).
\btitle{Asymptotic normality of the {$k$}-core in random graphs}.
\bjournal{Ann. Appl. Probab.}
\bvolume{18}
\bpages{1085--1137}.
\bid{doi={10.1214/07-AAP478}, issn={1050-5164}, mr={2418239}}
\end{barticle}
\bptok{imsref}%
\endbibitem

\bibitem{janson09b}
\begin{barticle}[mr]
\bauthor{\bsnm{Janson},~\bfnm{Svante}\binits{S.}} \AND
\bauthor{\bsnm{Luczak},~\bfnm{Malwina~J.}\binits{M.~J.}}
(\byear{2009}).
\btitle{A new approach to the giant component problem}.
\bjournal{Random Structures Algorithms}
\bvolume{34}
\bpages{197--216}.
\bid{doi={10.1002/rsa.20231}, issn={1042-9832}, mr={2490288}}
\end{barticle}
\bptok{imsref}%
\endbibitem

\bibitem{jlr}
\begin{bbook}[mr]
\bauthor{\bsnm{Janson},~\bfnm{Svante}\binits{S.}},
\bauthor{\bsnm{{\L}uczak},~\bfnm{Tomasz}\binits{T.}} \AND
\bauthor{\bsnm{Rucinski},~\bfnm{Andrzej}\binits{A.}}
(\byear{2000}).
\btitle{Random Graphs}.
\bpublisher{Wiley},
\blocation{New York}.
\bid{doi={10.1002/9781118032718}, mr={1782847}}
\end{bbook}
\bptok{imsref}%
\endbibitem

\bibitem{kallenberg}
\begin{bbook}[mr]
\bauthor{\bsnm{Kallenberg},~\bfnm{Olav}\binits{O.}}
(\byear{2002}).
\btitle{Foundations of Modern Probability},
\bedition{2nd} ed.
\bpublisher{Springer},
\blocation{New York}.
\bid{doi={10.1007/978-1-4757-4015-8}, mr={1876169}}
\end{bbook}
\bptok{imsref}%
\endbibitem

\bibitem{kesten86}
\begin{bincollection}[mr]
\bauthor{\bsnm{Kesten},~\bfnm{Harry}\binits{H.}}
(\byear{1986}).
\btitle{Aspects of first passage percolation}.
In \bbooktitle{\'{E}cole D'\'et\'e de Probabilit\'es de {S}aint-{F}lour, {XIV}---1984}.
\bseries{Lecture Notes in Math.}
\bvolume{1180}
\bpages{125--264}.
\bpublisher{Springer},
\blocation{Berlin}.
\bid{doi={10.1007/BFb0074919}, mr={0876084}}
\end{bincollection}
\bptok{imsref}%
\endbibitem

\bibitem{KM2010}
\begin{barticle}[mr]
\bauthor{\bsnm{Klenke},~\bfnm{Achim}\binits{A.}} \AND
\bauthor{\bsnm{Mattner},~\bfnm{Lutz}\binits{L.}}
(\byear{2010}).
\btitle{Stochastic ordering of classical discrete distributions}.
\bjournal{Adv. in Appl. Probab.}
\bvolume{42}
\bpages{392--410}.
\bid{doi={10.1239/aap/1275055235}, issn={0001-8678}, mr={2675109}}
\end{barticle}
\bptok{imsref}%
\endbibitem

\bibitem{lel-diff}
\begin{barticle}[mr]
\bauthor{\bsnm{Lelarge},~\bfnm{Marc}\binits{M.}}
(\byear{2012}).
\btitle{Diffusion and cascading behavior in random networks}.
\bjournal{Games Econom. Behav.}
\bvolume{75}
\bpages{752--775}.
\bid{doi={10.1016/j.geb.2012.03.009}, issn={0899-8256}, mr={2929480}}
\end{barticle}
\bptok{imsref}%
\endbibitem

\bibitem{MolReed98}
\begin{barticle}[mr]
\bauthor{\bsnm{Molloy},~\bfnm{Michael}\binits{M.}} \AND
\bauthor{\bsnm{Reed},~\bfnm{Bruce}\binits{B.}}
(\byear{1998}).
\btitle{The size of the giant component of a random graph with a given degree sequence}.
\bjournal{Combin. Probab. Comput.}
\bvolume{7}
\bpages{295--305}.
\bid{doi={10.1017/S0963548398003526}, issn={0963-5483}, mr={1664335}}
\end{barticle}
\bptok{imsref}%
\endbibitem

\bibitem{riowor10}
\begin{barticle}[mr]
\bauthor{\bsnm{Riordan},~\bfnm{Oliver}\binits{O.}} \AND
\bauthor{\bsnm{Wormald},~\bfnm{Nicholas}\binits{N.}}
(\byear{2010}).
\btitle{The diameter of sparse random graphs}.
\bjournal{Combin. Probab. Comput.}
\bvolume{19}
\bpages{835--926}.
\bid{doi={10.1017/S0963548310000325}, issn={0963-5483}, mr={2726083}}
\end{barticle}
\bptok{imsref}%
\endbibitem

\bibitem{HHM05}
\begin{barticle}[mr]
\bauthor{\bsnm{van~der Hofstad},~\bfnm{Remco}\binits{R.}},
\bauthor{\bsnm{Hooghiemstra},~\bfnm{Gerard}\binits{G.}} \AND
\bauthor{\bsnm{Van Mieghem},~\bfnm{Piet}\binits{P.}}
(\byear{2005}).
\btitle{Distances in random graphs with finite variance degrees}.
\bjournal{Random Structures Algorithms}
\bvolume{27}
\bpages{76--123}.
\bid{doi={10.1002/rsa.20063}, issn={1042-9832}, mr={2150017}}
\end{barticle}
\bptok{imsref}%
\endbibitem

\end{thebibliography}
%
%



\printaddresses
\end{document}